%% file: main.tex
\documentclass[11pt]{amsart}
\usepackage[margin=1in]{geometry}

\usepackage{amssymb}
\usepackage{amsthm}
\usepackage{amsmath}
\usepackage{mathrsfs}
\usepackage{amsbsy}
\usepackage{standalone}
\usepackage[all]{xy}
\usepackage{bm}
\usepackage{hyperref}
\usepackage{tikz}
\usetikzlibrary{quotes}
\usepackage{array}
\usepackage{enumerate}
\usepackage{xcolor}
\usepackage{bbm}
\usepackage{comment}
\usepackage{dynkin-diagrams}
\usepackage{ifthen}
\usepackage{thmtools, thm-restate}
\graphicspath{{./}{figures/}}

\usepackage[noabbrev,capitalise,nameinlink]{cleveref}

\definecolor{lavender}{rgb}{0.4,0,1.0}
\definecolor{peach}{rgb}{1,0.43,0.39}
\definecolor{DarkPink}{rgb}{0.9,0,0.45} 
\definecolor{Green}{RGB}{0,204,0} 

\hypersetup{colorlinks=true, citecolor=DarkPink, linkcolor=DarkPink,urlcolor=DarkPink}

\usepackage{hhline}
\allowdisplaybreaks
\usepackage[noadjust]{cite}
\usepackage{asymptote}

\usepackage{caption}
\usepackage{diagbox}
\usepackage[noabbrev,capitalise,nameinlink]{cleveref}
\crefname{conjecture}{Conjecture}{Conjectures}

\newtheorem{theorem}{Theorem}[section]
\newtheorem{proposition}[theorem]{Proposition}
\newtheorem{corollary}[theorem]{Corollary}
\newtheorem{conjecture}[theorem]{Conjecture}

\newtheorem{lemma}[theorem]{Lemma}

\theoremstyle{definition}
\newtheorem{definition}[theorem]{Definition}
\newtheorem{remark}[theorem]{Remark}
\newtheorem{observation}[theorem]{Observation}
\crefname{observation}{Observation}{Observations}
\newtheorem{example}[theorem]{Example}

\newcommand{\Inv}{\mathrm{Inv}} 
 
\newcommand{\Cay}{\mathrm{Cay}}
\newcommand{\Cox}{\mathrm{Cox}}

\newcommand{\wh}{\widehat}

\newcommand{\id}{\mathrm{id}} 
\newcommand{\yy}{\mathtt{src}}

\newcommand{\ip}[2]{\langle #1 , #2 \rangle }
\renewcommand{\AA}{\mathbb{A}}

\newcommand{\DD}{\mathbb{D}}

\newcommand{\hyp}{\mathbb{H}}

\newcommand{\NN}{\mathbb{N}}

\newcommand{\RR}{\mathbb{R}}

\newcommand{\Span}{\mathrm{Span}}

\newcommand{\conv}{\mathrm{Conv}}

\newcommand{\abs}[1]{|#1|}
\newcommand{\fund}{D} 

\newcommand{\cC}{\mathcal{C}}

\newcommand{\cS}{\mathcal{S}}

\newcommand{\PH}{\mathrm{PH}}
\newcommand{\cw}{\mathrm{cw}}
\newcommand{\arcc}{\mathrm{arc}}
\newcommand{\resroots}{\widehat{\Phi}}
\newcommand{\pgon}{P}
\newcommand{\polar}{P^\diamond}

\DeclareFontFamily{U}{mathx}{}
\DeclareFontShape{U}{mathx}{m}{n}{<-> mathx10}{}
\DeclareSymbolFont{mathx}{U}{mathx}{m}{n}
\DeclareMathAccent{\widehat}{0}{mathx}{"70}
\DeclareMathAccent{\widecheck}{0}{mathx}{"71}

\newcommand{\dfn}[1]{\textcolor{lavender}{\emph{#1}}}

\begin{document}

\title[Extended weak order]{Extended weak order in all rank three Coxeter groups}
\subjclass[2020]{}

\author[]{Grant Barkley} \address[GB]{Department of Mathematics,
  University of Michigan, Ann Arbor, MI 48109, USA}
\email{gbarkley@umich.edu}

\author[]{Colin Defant} \address[CD]{Department of Mathematics,
  Harvard University, Cambridge, MA 02138, USA}
\email{colindefant@gmail.com}

\author[]{Patricia Hersh} \address[PH]{Department of Mathematics,
  University of Oregon, Eugene, OR 97403, USA}
\email{plhersh@uoregon.edu}

\author[]{Jon McCammond} \address[JM]{Department of Mathematics,
  University of California, Santa Barbara, CA 93106, USA}
\email{jon.mccammond@math.ucsb.edu}

\author[]{Thomas McConville} \address[TM]{Department of Mathematics,
  Kennesaw State University, Marietta, GA 30060, USA}
\email{tmcconvi@kennesaw.edu}

\author[]{David E Speyer} \address[DES]{Department of Mathematics,
  University of Michigan, Ann Arbor, MI 48109, USA}
\email{speyer@umich.edu} 

\begin{abstract} 
  We prove that extended weak order is a lattice for all rank~$3$ 
  Coxeter groups.  This was previously known in finite and affine 
  types, and in the universal Coxeter group of rank $3$. All other 
  cases are new to this paper. As in earlier work, the proof proceeds 
  by showing that all rank 3 Coxeter arrangements are clean arrangements, 
  meaning that every biclosed set is weakly separated from its 
  complement by a hyperplane. As a consequence, we also deduce Dyer's 
  Conjecture A for all rank 3 Coxeter groups, which asserts that cover 
  relations in extended weak order are given by toggling a single root.
\end{abstract}

\maketitle

\section{Introduction}\label{sec:intro}
Let $(W,S)$ be a Coxeter system. Much of the structure of this 
system is most naturally seen through the associated root system 
$\Phi$. The roots in $\Phi$ record the reflecting hyperplanes, 
organize the action of the group on its standard linear representation, 
and provide a bridge between the combinatorics of words and the 
geometry of arrangements.  

Each element of $W$ has an inversion set, which is a subset of 
the set $\Phi^+$ of positive roots. The \dfn{weak order} is the 
partial order on $W$ defined by containment of inversion sets. This 
order is especially well behaved when $W$ is finite; in particular, 
it is a lattice. This lattice structure is fundamental in Coxeter 
combinatorics. For example, it underlies the classical Garside structure 
on spherical Artin groups \cite{BS72}. It is also the starting point 
for lattice quotients such as Cambrian and Tamari lattices, which 
connect Coxeter groups to generalized associahedra and cluster 
combinatorics \cite{Reading06,ReadingSpeyer09,Reading12}. For infinite 
Coxeter groups, however, the situation changes.  Bj\"orner proved that 
weak order is always a meet-semilattice \cite{Bjorner84}, but it does 
not have joins in general. 

Dyer introduced the extended weak order as a root-theoretic enlargement 
of weak order, motivated in part by questions arising from Kazhdan--Lusztig 
theory \cite{dyer:1994quotients,dyer:2019weak}.  A subset $B\subseteq 
\Phi^+$ is called \dfn{closed} if, whenever $\alpha,\beta\in B$, every 
positive root in the nonnegative span of $\alpha$ and $\beta$ also lies 
in $B$.  We say $B$ is \dfn{biclosed} if both $B$ and $\Phi^+\setminus B$ 
are closed.  The \dfn{extended weak order} is the collection of biclosed 
subsets of $\Phi^+$, ordered by inclusion.  The map sending each element of 
$W$ to its inversion set embeds the ordinary weak order as an order ideal 
in this larger poset. 

We are primarily interested in a conjecture of Dyer, which states that the 
extended weak order is a lattice for every Coxeter group. If true, this 
lattice structure could provide a mechanism for extending some of the notable 
consequences of the lattice property in finite types to infinite types. For 
example, Barkley and Speyer proved Dyer's conjecture in all affine types 
\cite{barkley:2023affine}, and there has now been subsequent work developing 
lattice quotients of extended weak order in this setting 
\cite{BarkleyThesis, BarkleyDefant25}. 

Dyer's conjecture is known in several important cases.  In finite type, the
extended weak order coincides with the usual weak order, so the conjecture
follows from the classical lattice property.  Wang proved the conjecture for
affine Coxeter groups of rank $3$ \cite{wang:2019infinite}, and Barkley and 
Speyer proved it for all affine Coxeter groups \cite{barkley:2023affine}.  
In our previous paper \cite{BDHMMS25}, we proved Dyer's conjecture for $U_3$, 
the universal Coxeter group of rank $3$.  This was the first non-spherical, 
non-affine Coxeter group for which Dyer's conjecture was proven.  More precisely, 
we proved that every biclosed set of positive roots for $U_3$ is weakly separable, 
and we used this to show that the extended weak order for $U_3$ is a lattice.  
We also established a refined description of joins in that lattice.

The present paper establishes Dyer's conjecture for all Coxeter groups of 
rank $3$.  In fact, for these groups, we prove the stronger fact that biclosed 
sets are weakly separable. 

\begin{theorem}\label{thm:main}
  Let $(W,S)$ be a Coxeter system of rank $3$, and let $\Phi^+$ be its set of 
  positive roots.  Every biclosed subset of $\Phi^+$ is weakly separable.
\end{theorem} 

Since the poset of weakly separable sets in a rank $3$ root system is known 
to be a lattice \cite{LabbeThesis,barkley:2023affine},  \cref{thm:main} 
immediately implies: 

\begin{theorem}\label{thm:main2}
  Let $(W,S)$ be a Coxeter system of rank $3$, and let $\Phi^+$ be its set 
  of positive roots. The poset of biclosed sets in $\Phi^+$ is a lattice.
\end{theorem}

\cref{thm:main,thm:main2} appear in the main body of the paper as 
\Cref{thm:separable} and \Cref{cor:lattice}.

Let us note that in ranks greater than 3, weak separability and biclosedness 
can be genuinely different conditions. See, for example, 
\cite[Example 6.13]{barkley:2024combinatorial}. Also, it is known that in 
ranks greater than 3, weakly separable sets do not form a lattice, even though 
this is not currently in the literature. For example, the join computed 
in \cite[Example 6.13]{barkley:2024combinatorial} does not exist in the poset 
of weakly separable sets.

We now sketch the proof of \Cref{thm:main}, which occupies the main body of 
the paper. Let $B \subseteq \Phi^+$ be a biclosed set and put $R = 
\Phi^+ \setminus B$. The cases where $B$ or $R$ are finite  are already known, 
so we focus on the case where $B$ and $R$ are both infinite. We visualize $B$ 
and $R$ as collections of dots in a plane, colored blue and red, and our job is 
to draw a line that separates one color of dot from the other. We ignore the 
issue of dots on the line for the sake of this sketch.

The set of accumulation points of our dots forms an ellipse, the boundary of 
the Tits cone. We color each point of this ellipse either red, blue or purple, 
according to whether it is an accumulation point of red dots, of blue dots, or 
of both. Our key lemma, \Cref{boundary-coloring-main}, is that the boundary is 
separated into two arcs, one of which is red and one of which is blue, with the 
boundary points being purple. Our dividing line is then the line through the two 
purple points. 

\cref{fig:ProofMethod} illustrates this method: We have depicted all roots of 
depth $\leq 15$ for the triangle group $(2,3,7)$ and colored them red and blue 
in a biclosed manner. The accumulation points of the roots form an ellipse, and 
the coloring of the roots gives rise to a coloring of the ellipse with a single 
red arc and a single blue arc. The straight line passes through the endpoints of 
those arcs, and separates the red roots from the blue roots.

Thus far, our summary also describes the proof in our previous paper~\cite{BDHMMS25}.
We now explain what new ideas were needed in this paper. In the previous paper, 
we defined a bi-infinite path, called the snake, through the Cayley graph of our 
Coxeter group, and we showed that the endpoints of this snake approach the two 
purple points in \Cref{boundary-coloring-main}. In the current setting, we were 
unable to build a single snake with this property.

Instead, we identify a discrete property of our coloring that implies 
\Cref{boundary-coloring-main}: We introduce a certain sort of convex polygon with 
vertices at the roots, which we term a \emph{root polygon}. The Polygon Lemma  
(\Cref{lem:polygon-lemma}) shows that, when we color the vertices of a root polygon 
according to $R$ and $B$, the boundary of the polygon is divided into two monochromatic 
arcs. We can then approximate the boundary of the Tits cone by appropriate polygons 
in order to prove \Cref{boundary-coloring-main}.

In order to prove the Polygon Lemma, we use a variant of our snake construction 
from~\cite{BDHMMS25}, but we use a snake specifically adapted to the polygon in 
question, rather than trying to find a single snake that works for the whole proof. 
Thus, the main new technical tools in this paper are the introduction of root polygons 
in \Cref{sec:rootpolygon} and the Snake Nest construction from \Cref{sec:polygonlemma}, 
which is used to prove the Polygon Lemma in the same section. \Cref{sec:boundarycolor} 
contains the limiting argument, \Cref{sec:FinishProof} completes the proof of the main 
theorem, and \Cref{sec:Consequences} spells out the consequences of the main theorem. 

We note that most of our results from \cref{sec:rootpolygon,sec:polygonlemma,sec:boundarycolor} 
apply more generally to the \emph{parabolic biclosed sets} introduced in \cite{BDHMMS25}; 
see \Cref{dfn:parabolic_biclosed}. It is only in the last stages of the proof that we need 
to use the full strength of the biclosed property.

\begin{figure}[t]
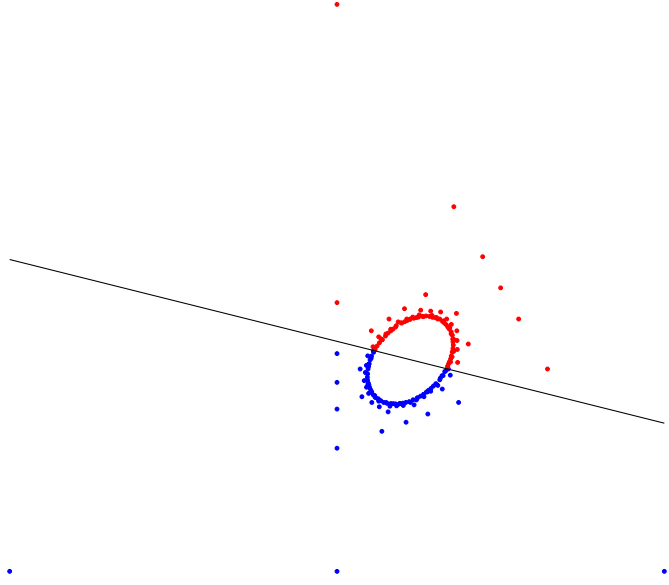

  \centering 
  \includestandalone{rootFig}
  \caption{The roots of the $(2,3,7)$ Coxeter group, colored in a biclosed fashion, 
  and the line that separates them.  The convex hull of the roots is a triangle 
  whose extremal vertices are the simple roots.  They surround the disk $\DD$,
  which in this affine patch is an ellipse.} 
  \label{fig:ProofMethod}
\end{figure}

\section{Background: Rank $3$ Coxeter groups and the hyperbolic
plane}\label{sec:background_geometry}
This section highlights how almost all rank~$3$ Coxeter groups act 
on the hyperbolic plane and it establishes notation for the remainder 
of the article.  The general approach is analogous to that taken in our 
earlier work \cite{BDHMMS25}, but modified to allow for arbitrary rank~$3$ 
Coxeter groups.

\begin{definition}[Rank $3$ Coxeter Groups]\label{def:rank3}
  The \dfn{rank~$3$ Coxeter group} $\Cox(p,q,r)$ is the group defined 
  by the presentation
  \[ 
  \Cox(p,q,r) = \langle s_1, s_2, s_3 \mid s_1^2, s_2^2, s_3^2, (s_1 s_2)^p, 
  (s_1 s_3)^q, (s_2 s_3)^r \rangle
  \]
  with $p,q,r \in \{2,3,\ldots\} \cup \{\infty\}$.  When one or more of 
  $p,q,r$ is $\infty$, no relation is added, and to avoid redundancies we 
  assume $p \leq q \leq r$.  For example, the Coxeter group $\Cox(2,3,\infty)$ 
  has presentation 
  \[
  \Cox(2,3,\infty)=\langle s_1,s_2,s_3 \mid s_1^2, s_2^2, s_3^2, (s_1 s_2)^2, 
  (s_1 s_3)^3\rangle.
  \]     
  When $p=q=2$, the group $\Cox(2,2,r)$ is \dfn{reducible} in the sense 
  that it splits as a direct product of a rank~$1$ Coxeter group generated by 
  $s_1$ and a rank~$2$ Coxeter group generated by $s_2$ and $s_3$.  The other 
  rank~$3$ Coxeter groups are \dfn{irreducible}. 
\end{definition}

\begin{definition}[Parabolic Subgroups]
  Let $W$ be a Coxeter group with generating set $S$.  A \dfn{standard parabolic 
  subgroup} of $W$ is $\langle S' \rangle$, the subgroup generated by $S' \subset S$.
  A \dfn{parabolic coset} is a left coset $w\cdot \langle S'\rangle$ of a standard 
  parabolic subgroup, and a (general) \dfn{parabolic subgroup} is a conjugate 
  subgroup $w\cdot \langle S'\rangle \cdot w^{-1}$. For an irreducible rank~$3$ 
  Coxeter group such as $\Cox(p,q,r)$, there are seven proper standard parabolic 
  subgroups corresponding to the seven proper subsets of $\{s_1,s_2,s_3\}$, including 
  $\emptyset$. For the empty set, the standard parabolic is the trivial subgroup and 
  its left cosets are the individual elements of $W$.
\end{definition}

\begin{definition}[The Cayley graph] \label{def:Cayley}
  The (right) \dfn{Cayley graph} of the group $W = \Cox(p,q,r)$ with 
  generating set $S = \{s_1, s_2, s_3\}$ is a graph $\Gamma=\Cay(W,S)$ 
  with vertices $v_w$ indexed by $w \in W$ and directed edges $e_{w,s}$ 
  from $v_w$ to $v_{w\cdot s}$ indexed by $W \times S$.  The \dfn{label} 
  of the edge is the generator $s\in S$ and its \dfn{endpoints} are 
  $\{w,w\cdot s\}$, which is the same as $w \langle s\rangle$ since the 
  generators are involutions.  The \dfn{length} of an element in $W$ is 
  the length of the shortest word for $w$ in the generators $s_1$, $s_2$, 
  $s_3$, and written $\ell(w)$. Edges of the Cayley graph always join 
  group elements whose lengths differ by $\pm 1$, and we often draw these 
  edges oriented from the element of lesser length to the element of greater 
  length. 
\end{definition}

Although $p,q,r$ are allowed to be infinity, on first reading readers may 
wish to assume that all three parameters are finite since this simplifies 
several technical details. For example, the situation described in 
\cref{lem:cusps} is impossible when $p$, $q$ and $r$ are finite, and we will 
not need to use the double party hats introduced in \cref{def:doublepartyhat}.
Irreducible rank~$3$ Coxeter groups act on triangular tilings in constant 
curvature spaces.

\begin{definition}[Triangular Tilings]\label{def:tritile}
  Every irreducible Coxeter group $\Cox(p,q,r)$ acts on a space of constant 
  curvature depending on the value of $m = \frac{1}{p} + \frac{1}{q} + \frac{1}{r}$.
  When $m$ is greater than, equal to, or less than $1$, there is an essentially 
  unique spherical, Euclidean, or hyperbolic triangle with angles $\frac{\pi}{p}$, 
  $\frac{\pi}{q}$, and $\frac{\pi}{r}$, and reflections in the sides of this 
  triangle generate a \dfn{reflection group} isomorphic to $\Cox(p,q,r)$, and 
  it subdivides the space into a \dfn{triangular tiling}.  See \cref{fig:Klein} 
  for the triangular tiling for the reflection group $\Cox(2,3,7)$ shown in the 
  Klein disk model of the hyperbolic plane.\footnote{\cref{fig:Klein} does not 
  use the affine patch from \cref{def:klein} and, also, some of the roots are 
  off the page. Indeed, if a hyperplane passes through the exact center $O$ of 
  the disk, its root is in the plane $O^\bot$, and the rescaled root is ``at 
  infinity". \Cref{fig:ProofMethod} depicts the same root system in a different 
  affine patch; the circle in \cref{fig:Klein} is the interior of the ellipse in 
  \Cref{fig:ProofMethod}.}
\end{definition}

\begin{figure}
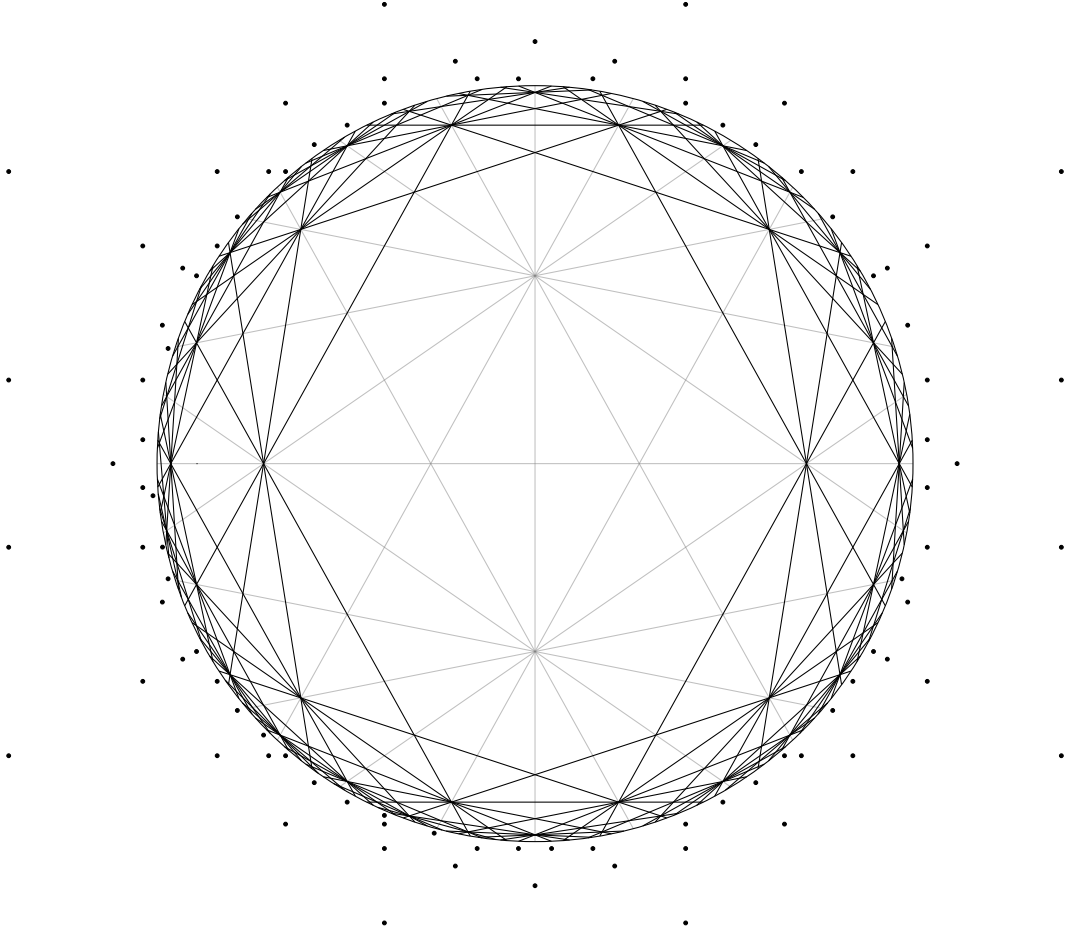

  \includestandalone{Kleinfigure}
  \caption{The triangular tiling for the reflection group $\Cox(2,3,7)$ in the 
  Klein model of the hyperbolic plane $\DD$. A portion of the root system $\Phi$ 
  is shown as well. Some roots are off the page, and hyperplanes dual to those 
  roots are shown in gray.}\label{fig:Klein}
\end{figure}
 
\begin{remark}[Mostly Hyperbolic]\label{rem:mostlyhyperbolic}
  The irreducible rank~$3$ Coxeter groups $\Cox(p,q,r)$ are spherical 
  when $(p,q,r) = (2,3,k)$ where $k=3,4,5$, and Euclidean when $(p,q,r) 
  = (3,3,3)$, $(2,4,4)$ or $(2,3,6)$. The former are the Coxeter groups 
  of type $A_3$, $B_3$ and $H_3$, and the latter are the Coxeter groups 
  of type $\widetilde A_2$, $\widetilde B_2 = \widetilde C_2$, and 
  $\widetilde G_2$.  In all other cases $\Cox(p,q,r)$ is hyperbolic.
  The \dfn{universal rank~$3$ Coxeter group} studied in our previous 
  article~\cite{BDHMMS25} is $U_3 = \Cox(\infty,\infty,\infty)$.
\end{remark}

At this point we exclude the $3$~spherical cases and the 
$3$~Euclidean cases to focus exclusively on the irreducible hyperbolic 
rank~$3$ Coxeter groups and how they act on the hyperbolic plane.

\begin{definition}[Hyperbolic Plane]\label{def:hyp-plane}
  Let $V$ be a $3$-dimensional real vector space equipped with a quadratic 
  form of signature $(+,+,-)$. The corresponding \dfn{bilinear form} is 
  $\ip{x}{y} = \frac12 (Q(x+y)-Q(x)-Q(y))$.  The value $Q(x) = \ip{x}{x}$
  is the \dfn{norm} of $x$.  Let $O(2,1)$ be the \dfn{indefinite orthogonal 
  group} of linear transformations of $V$ that preserve the bilinear form.   
  The double cone of norm $0$ vectors is the \dfn{isotropic cone}. The 
  vectors on the lines inside the isotropic cone have negative norm and 
  those outside the isotropic have positive norm.     
    
  The quadratic form $Q$ defines a hyperboloid of two sheets $\{x \mid Q(x)=-1 \}$. 
  If the quadratic form is diagonalized as $Q(x_1, x_2, x_3) = x_1^2+x_2^2-x_3^2$, 
  then one of those sheets has $x_3 \geq 1$ and the other has $x_3 \leq -1$. 
  Either of those sheets forms a model of the hyperbolic plane, called the 
  \dfn{hyperboloid model} and denoted $\hyp^2$. The points of $\hyp^2$ give 
  unique representatives for the negative norm lines through the origin.
  The  essentially unique $O(2,1)$-invariant metric on $\hyp^2$ is given 
  by the formula
  \[
  d(x,y) = \cosh^{-1} \abs{\ip{x}{y}}.
  \]    
\end{definition}

The symmetries of the hyperbolic plane are generated by reflections in 
positive norm vectors.

\begin{definition}[Reflections]
  A \dfn{reflection} is a linear transformation of order $2$ that fixes 
  a hyperplane pointwise.  For a vector $v$ with $Q(v)>0$, we define a 
  reflection $r_v\colon V\to V$ by 
  \[
    r_v(x) = x - 2 \frac{\ip{v}{x}}{\ip{v}{v}}v.
  \]
  This fixes the plane $H_v = \{x \colon \ip{v}{x}=0\}$ pointwise and 
  swaps the two open halfspaces $H_v^+ := \{x \colon \ip{v}{x}>0\}$ and  
  $H_v^- := \{x \colon \ip{v}{x}<0\}$. We will write $\overline{H^+}$ and 
  $\overline{H^-}$ for the corresponding closed halfspaces.
  Reflections in negative norm vectors also exist, but they are not relevant 
  since they swap the two sheets of the hyperboloid $\{ x \mid Q(x)=-1 \}$. 
\end{definition}

The representation of $\Cox(p,q,r)$ as a reflection group in $O(2,1)$ 
starts by selecting simple roots.

\begin{definition}[Roots and the fundamental chamber]\label{def:roots}
  If $p$, $q$, and $r$ define a hyperbolic rank $3$ Coxeter group, then 
  there exist vectors $\Delta = \{\alpha_1, \alpha_2, \alpha_3\}$ in 
  $\RR^{2,1}$, such that
  \[ 
  \langle \alpha_i,\alpha_j\rangle = 
  \begin{cases}
     1\ &\mbox{if }i=j\\
     -\cos(\frac{\pi}{p})\ &\mbox{if }\{i,j\}=\{1,2\}\\
     -\cos(\frac{\pi}{q})\ &\mbox{if }\{i,j\}=\{1,3\}\\
     -\cos(\frac{\pi}{r})\ &\mbox{if }\{i,j\}=\{2,3\}\\
  \end{cases}.
  \]
  Here $-\cos \tfrac{\pi}{\infty}$ is treated as $-1$; see 
  \cref{rem:finite-area}. The vectors $\alpha_i$ are the 
  \dfn{simple roots} of $\Cox(p,q,r)$ and the reflections $r_{\alpha_i}$
  generate a subgroup of $O(2,1)$ isomorphic to the Coxeter group $\Cox(p,q,r)$, 
  with $s_i$ acting by $r_{\alpha_i}$.  The orbit of $\Delta := \{\alpha_i\}$ 
  under the action of $\Cox(p,q,r)$ is the \dfn{root system} $\Phi$.  
  Note that if $\Phi' \subset \Phi$ is the subset of roots of the standard 
  parabolic subgroup $\langle S'\rangle$, then $w \cdot \Phi'$ is the 
  subset of roots of the parabolic subgroup $w\cdot \langle S'\rangle \cdot w^{-1}$ 
  with $r_{w\cdot \alpha_i} = w\cdot r_{\alpha_i} \cdot w^{-1}$.

  The simplicial cone $H_{\alpha_1}^- \cap H_{\alpha_2}^- \cap H_{\alpha_3}^-$ 
  intersects one of the two sheets of the hyperboloid 
  $\{ x : \langle x,x \rangle = -1 \}$. 
  We choose our hyperboloid model $\hyp^2$ to be the sheet that intersects this cone.
  The \dfn{fundamental chamber}, $\fund$ is $\overline{H_{\alpha_1}^-} \cap
  \overline{H_{\alpha_2}^-} \cap \overline{H_{\alpha_3}^-} \cap \hyp^2$.  
  This is a (possibly ideal) finite area hyperbolic triangle. The triangles 
  $w \cdot \fund$, 
  for $w \in \Cox(p,q,r)$, form the triangular tiling discussed in \Cref{def:tritile}, 
  with a distinct triangle for each $w \in \Cox(p,q,r)$. 
   
  There is a unique point $O$ in the fundamental chamber with 
  $\langle \alpha_1, O \rangle = \langle \alpha_2, O \rangle = 
  \langle \alpha_3, O \rangle$; this common value is negative. 
  (Geometrically, $O$ is the incenter of the hyperbolic triangle 
  $\fund$.) The roots in $\Phi$ with negative/positive 
  values of $\langle O, \ \rangle$ are the \dfn{positive/negative roots}. 
  The sets of positive and negative roots are denoted $\Phi^+$ and $\Phi^-$, 
  and we have $\Phi = \Phi^+ \sqcup \Phi^-$.  Every root $\alpha \in \Phi^+$ 
  is a non-negative linear combination of the simple roots $\alpha_i \in \Delta$.
\end{definition} 

As in the previous article \cite{BDHMMS25}, we prefer to work in the Klein model 
in an affine patch. Readers may be familiar with pictures of the hyperbolic plane 
where reflecting hyperplanes are drawn as arcs of circles orthogonal to the boundary 
of the hyperbolic plane, such as Escher's \emph{Circle Limit IV (Heaven and Hell)}.
These are drawn in a different model, the Poincar\`e model, which we do not use in 
this paper. In our paper we use the Klein model, so reflecting hyperplanes are drawn 
as line segments.

\begin{definition}[The affine model and the Klein model]\label{def:klein}
  The standard \dfn{affine patch} is the affine plane $\AA = \{x \mid \ip{x}{O}=-1\} 
  \subset V$. This affine plane is tangent to $\hyp^2$ at $O$ and the metric 
  on $\hyp^2$ equips $\AA$ with a Euclidean norm.  Let $\DD$ be the disk 
  $\{ x \in \AA \mid Q(x) \leq 0 \}$ in $\AA$; this is the closed unit disk 
  centered at $O$ for the norm on $\AA$. The isotropic cone intersects $\AA$ 
  in the unit circle $\partial \DD$ and the negative norm lines (through the origin of $V$) 
  intersect $\AA$ in the interior of the unit disk $\DD^\circ$.  For any vector 
  $x$ in $V$ not in the plane $O^{\perp}$, the line $\RR x$ intersects the 
  affine patch in a unique \dfn{rescaled} point $\widehat x = \RR x  \cap \AA$.
  
  Since the root system $\Phi$ avoids the plane $O^{\perp}$, there is a 
  well-defined $2$-to-$1$ map $\Phi \to \widehat \Phi$ that sends $\pm \alpha \in \Phi$ 
  to $\widehat \alpha \in \widehat \Phi \subset \AA$, where $\widehat{\Phi}$ is 
  the image of $\Phi$ under this map. Note that these \dfn{rescaled roots} 
  $\widehat{\alpha} \in \widehat\Phi$ are in bijection with the positive 
  roots $\Phi^+$ and with the reflections in $\Cox(p,q,r)$.  The rescaled roots 
  in $\widehat{\Phi} \subset \AA$ are outside the closed disk $\DD$ (because 
  they have positive norm) and inside the triangle with (rescaled) vertices 
  $\widehat{\alpha_i}$, $i=1,2,3$ (since $\alpha \in \Phi^+$ is a non-negative 
  linear combination of the simple roots $\alpha_i$).
\end{definition}

\begin{remark}[Other affine patches]
  When making figures, it is sometimes convenient to replace the affine plane 
  $\AA$ in \cref{def:klein} with a different affine plane or different metric 
  where $\partial \DD$ may be an ellipse rather than a circle. In particular, 
  we have done this in \cref{fig:ProofMethod,fig:Klein}. We always use the 
  standard affine patch in proofs and computations.
\end{remark}

\begin{definition}[Point-Line Duality]\label{def:pt-line}
  There is a natural bijection between points $x$ in $\AA$ and lines $\RR x$ 
  through the origin in $V$ that do not lie in the plane $O^{\perp}$.  Similarly 
  there is a natural bijection between lines $L$ in $\AA$ and planes $H = \RR L$ 
  through the origin in $V$ other than the plane $O^{\perp}$ itself.  These 
  bijections, together with the quadratic form, produce a bijection between 
  the points in $\AA \setminus \{O\}$ and the lines in $\AA$ that do not pass 
  through $O$. 
  
  Concretely, a point $x \in \AA \setminus \{O\}$ determines a linear subspace 
  $\RR x$, whose orthogonal complement is a linear subspace $H_x = (\RR x)^\bot$, 
  which intersects $\AA$ in a line $L_x = H_x \cap \AA$ that does not pass through $O$.  
  More intrinsically, for $x \in \AA \setminus \{O\}$ we can define $L_x$ 
  and its halfspaces as follows:
  \begin{align*}
    L_x &= \{y \in \AA \mid \ \langle x,y\rangle = 0\},\\
    L_x^+ &= \{y \in \AA \mid \langle x,y\rangle > 0\}, \mbox{and}\\
    L_x^- &= \{y \in \AA \mid \langle x,y\rangle < 0\}.
  \end{align*}
  Geometrically, if $x \in \AA$ is outside $\DD$, then $L_x$ intersects the circle 
  $\partial \DD$ at the points of tangency of the two tangents to $\partial \DD$ 
  through $x$. If $x$ is on $\partial \DD$, then $L_x$ is the tangent line to 
  $\partial \DD$ at $x$. If $x\in\DD^\circ\setminus\{O\}$, then $L_x$ is entirely 
  outside $\DD$; in this last case, $L_x$ is the so-called ``polar line" to $x$. 
  The reader may enjoy taking a root in \cref{fig:Klein}, drawing the two tangents 
  to the circle $\partial \DD$ through that point, and observing that the points 
  of tangency are the endpoints of a reflecting hyperplane.
\end{definition}

\begin{remark}[$W$-action]
  Since the Coxeter group $W = \Cox(p,q,r)$ acts on the lines through the origin in 
  $\RR^3$ preserving norms, there is an induced incomplete action of $W$ on the points 
  and lines in $\AA$. The action on $\AA$ is incomplete because lines in $O^\bot$
  are not represented (\Cref{def:pt-line}), and the $W$-action sends lines of positive 
  norm into and out of $O^\bot$.  There are, however, $W$-actions on the points 
  inside $\DD$, on the points in $\partial \DD$, and on the rescaled roots $\widehat \Phi$.
\end{remark}

\begin{remark}[Inversion Sets]\label{rem:inversion}
  Let $x = v_w = \widehat{w\cdot O}$ be the image in $\AA$ of the center 
  of a triangle in the triangular tiling, i.e. a vertex of the Cayley graph. 
  The finitely many positive roots $\alpha$ with $\langle x,\alpha\rangle>0$ 
  form the \dfn{inversion set} $\Inv(w)$ of $w$.  They also correspond to the 
  finitely many lines / hyperplanes that separate $O$ from $w\cdot O$.
\end{remark}

The closure of the image of the fundamental chamber in $\AA$ is a 
triangle whose interior is completely contained in $\DD^\circ$ and there is 
a vertex in $\partial \DD$ if and only if there is an $\infty$ in $\{p,q,r\}$. 
When infinities do occur, our choice of form means that we have a finite area 
ideal hyperbolic triangle as a fundamental domain.

\begin{remark}[Finite Area Ideal Triangles]\label{rem:finite-area}
  When one or more of the parameters is $\infty$, our choice of quadratic form
  ensures that the fundamental domain for the action of $\Cox(p,q,r)$ on 
  $\hyp^2$ remains a triangle with finite area, but with one or more vertices 
  in $\partial \DD$, so that the sides of the triangle that do not meet 
  in the hyperbolic plane, meet at an \dfn{ideal point} in the boundary.  
  This also means that the reflections in these sides generate an infinite 
  dihedral group that acts on the hyperbolic plane fixing this ideal point 
  in the boundary. 
  
  In more general studies of root systems, when $s_i s_j$ has infinite order, 
  $\ip{\alpha_i}{\alpha_j}$ is allowed to be any number $\leq -1$. If at 
  least one such pairing is strictly less than $-1$, the fundamental domain 
  contains a positive length arc of $\partial \DD$, rather than just 
  isolated points and it has infinite area. Dyer has proved that the truth 
  of his main conjectures does not depend on the choice of 
  $\ip{\alpha_i}{\alpha_j}$ for $s_i s_j$ of infinite order;
  see \cref{independence2}.  However, many of our lemmas would 
  fail if we choose a more negative value for $\ip{\alpha_i}{\alpha_j}$, 
  starting with \Cref{accumulationpoints}.
\end{remark}

\begin{definition}[The Cayley graph drawn in the hyperbolic plane]
  We defined the Cayley graph of $\Cox(p,q,r)$ in \Cref{def:Cayley}. 
  We can embed it concretely in $\hyp^2$: The vertex $w$ is drawn at 
  $w \cdot O$ and each edge of the Cayley graph is a hyperbolic line 
  segment. The Cayley graph is thus the dual graph of the triangular 
  tiling, and we have drawn each edge of the Cayley graph to be orthogonal 
  to the dual edge in the tiling. 

  We now discuss the regions into which $\DD$ is divided by the Cayley graph. 
  Let $\langle s_i, s_j \rangle$ be a rank two parabolic subgroup of $W$ where 
  $s_i s_j$ has order $k \in \{ p,q,r \}$; let $w \cdot \langle s_i, s_j \rangle$ 
  be a left-coset of $\langle s_i, s_j \rangle$. Then $\{ g \cdot O : g \in w 
  \langle s_i, s_j \rangle \}$ forms a cycle or path in the Cayley graph -- 
  a cycle if $k < \infty$ and a path if $k = \infty$. These paths and cycles 
  are the boundaries of the disjoint convex regions of $\DD$ with the Cayley 
  graph removed. We call these the \dfn{faces of the Cayley graph}.
  The cycles bound \dfn{finite faces} and infinite paths bound \dfn{infinite faces}. 
  The finite faces are regular hyperbolic $2k$-gons. The infinite faces have 
  their vertices on a horocycle, which limits to an ideal point of the triangulation.
  (Here we are using that $\ip{\alpha_i}{\alpha_j} = -1$.)

  Recall that we use the length function to orient the Cayley graph. 
  We discuss the restriction of this orientation to $\partial F$, for a face $F$. 
  If $F$ is a finite $2k$-gon, then $\partial F$ has a unique source and a unique sink, 
  at distance $k$ apart. If $F$ is an infinite face, then $\partial F$ 
  has a unique source and no sink. 
\end{definition}

\begin{remark}[Faces and Linearly Ordered Roots]\label{rem:face-lines}
  Note that the set of Cayley graph edges bounding a single face $F$ 
  correspond to rescaled positive roots that all lie on a single line.  When $F$ 
  is a $2k$-face, the hyperplanes dual to the edges all pass through the 
  (hyperbolic) center of the $2k$-gon, and the roots lie on the hyperplane 
  dual to the center.  This case is illustrated in \cref{fig:From_Last_Paper}. 
  
  When $F$ is an infinite face, the hyperplanes dual to the edges all pass 
  through the ideal point of the face, and the roots lie on the hyperplane 
  dual to this point.  In $\AA$, the rescaled roots lie on the line tangent 
  to $\partial \DD$ at the ideal point.
  This case is illustrated in \Cref{fig:infinite_rainbow}.
  
  Moreover, in both cases, the two edges $e_1,e_2$ of $\partial F$ incident 
  to the unique source vertex  correspond to a pair of rescaled roots that 
  contains all of the other roots in their convex hull. As one proceeds along 
  edges of $\partial F$ from $e_i$ for $i\in \{ 1,2\}$ towards the sink of 
  $\partial F$ (or the ideal point in $\partial \DD$), the series of associated 
  roots of the edges one encounters sequentially themselves  appear sequentially 
  along their common line, which is disjoint from $\DD$ for a finite face and 
  tangent to $\partial \DD$ for an infinite face. The two choices of $i$ yield 
  sequences of roots proceeding from opposite sides.  For an $\infty$-face, 
  both sequences limit to the ideal point.
\end{remark}

\begin{remark}[Rank $2$ Subsystems]\label{rem:rank-2}
  Any two reflections $r_\alpha$ and $r_\beta$ in $W=\Cox(p,q,r)$ generate a 
  finite or infinite dihedral subgroup $G$. Geometrically, $G$ can be classified 
  as \dfn{finite}, \dfn{affine}, or \dfn{hyperbolic type}, depending on whether 
  the line $L_G$ determined by the rescaled roots $\widehat \alpha$ and $\widehat \beta$ 
  of its generating reflections intersects $\partial \DD$ in zero, one or two 
  points, respectively.  The finite case has finitely many roots and no 
  accumulation points.  In the affine and hyperbolic types, the intersection 
  $L_G \cap \partial \DD$ is exactly the set of accumulation points of the 
  rescaled roots in $G$.  
\end{remark}

\begin{figure}
  \centering
  \includegraphics[height=9.5cm]{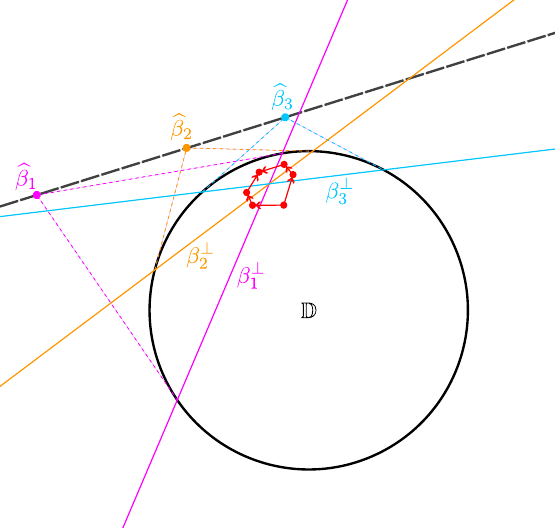}
  \caption{A hexagonal face of the Cayley graph and its three reflecting 
  lines. The reflecting lines meet at the hyperbolic center of the face, 
  and the corresponding rescaled roots lie on the polar line of this center. 
  The extreme rescaled roots, labeled $\beta_1$ and $\beta_3$, correspond 
  to the two edges incident to the source. }
  \label{fig:From_Last_Paper}
\end{figure} 

\begin{figure}
  \centering
  \includegraphics[height=10.512cm]{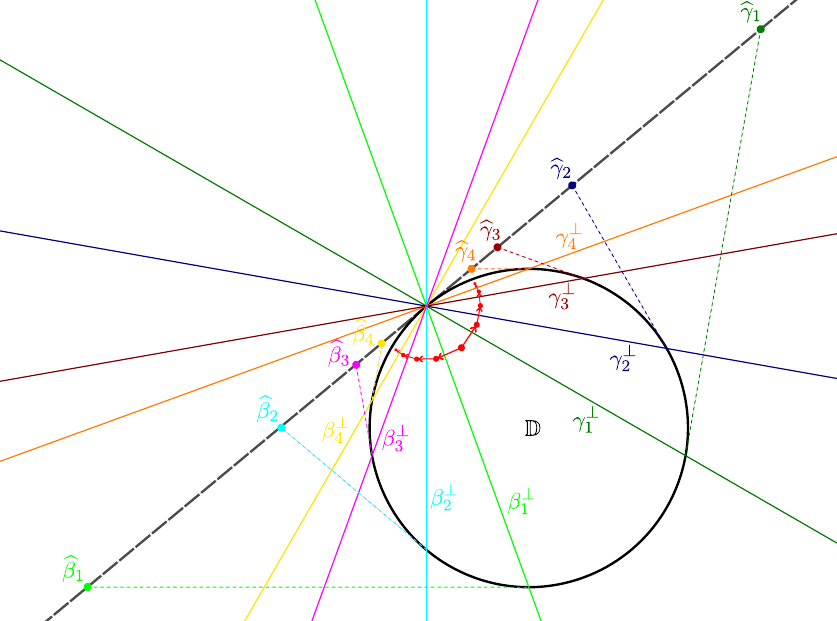}
  \caption{An infinite face of the Cayley graph and some of its reflecting 
  lines. The reflecting lines meet at the ideal point of the face, and the 
  corresponding rescaled roots, labeled $\beta_i$ and $\gamma_i$, lie on the 
  tangent line to $\partial\mathbb D$ at this point. The roots labeled 
  $\beta_1$ and $\gamma_1$ correspond to the two edges incident to the source. 
  Along the two boundary paths from the source toward the ideal point, the 
  associated rescaled roots converge to that point from opposite sides.}
  \label{fig:infinite_rainbow}
\end{figure}

The root system $\Phi$ is discrete as a subset of $V$, but the rescaled 
root system $\widehat \Phi$ in the affine patch $\AA$ contains limit points 
called \dfn{limit roots}. 

\begin{lemma}[Limit Roots]\label{accumulationpoints}
  The set $\widehat \Phi$ is discrete as a subset of $\AA \setminus \DD$, 
  but its closure in $\AA$ includes all of $\partial \DD$, i.e. every point 
  in $\partial \DD$ is the limit of points in $\widehat \Phi^+$ and these 
  are the only limit points.
\end{lemma}

\begin{proof}
  Let $S = \{ x : Q(x)=1\}$ be hyperboloid of one sheet containing $\Phi$.
  The natural map from $\RR^3$ minus the plane $O^\bot$ to $\AA$ restricts to 
  a $2$-sheeted cover from $S$ minus the unit circle (in the $O^\bot$-plane) to 
  $\AA\setminus\DD$, and since $\Phi$ is a discrete subset of $S$, $\wh{\Phi}$ 
  is a discrete subset in $\AA\setminus\DD$. The second fact is Theorem 4.4 of 
  \cite{DyerHohlwegRipoll} since, by our choice of form, the fundamental chamber 
  has finite hyperbolic area.
\end{proof}

\begin{corollary}[Red, Blue or both]\label{cor:red-blue}
  For every disjoint partition $\widehat \Phi = R \sqcup B$, every point in 
  $\partial \DD$ is a limit of the rescaled roots in $R$, a limit of the rescaled 
  roots in $B$, or both.
\end{corollary}

\begin{proof}
  Every point $x \in \partial \DD$ is the limit of a sequence of distinct 
  rescaled roots in $\widehat \Phi$ and this sequence has infinitely many 
  elements in $R$ or in $B$ or both.
\end{proof}

\section{Background: Extended Weak Order}\label{sec:bkgnd-weak-order}

In this section, we recall the definition of the extended weak order 
for any Coxeter system and present Dyer's conjectures. Throughout this 
section, $(W,S)$ is a finitely generated Coxeter system. We choose a 
geometric root system $\Phi$ with simple roots $\Delta$ following the 
same conventions as in \cref{sec:background_geometry}.

\begin{definition}[Inversion sets, weak order]
  For $\alpha\in\Phi^+,\ w\in W$, we say $\alpha$ is an \dfn{inversion} 
  of $w$ if $w^{-1}\cdot \alpha\in\Phi^-$. The \dfn{inversion set} of $w$, 
  denoted $\Inv(w)$, is its set of inversions. The \dfn{weak order} is 
  the poset $(W,\leq)$ in which $u\leq v$ if $\Inv(u)\subseteq \Inv(v)$.
\end{definition}

\begin{definition}[Reduced words]
  Any element of $W$ can be written as a product of simple reflections.
  If $w=s_1\cdots s_k$ for some $s_i\in S$, we say $s_1\cdots s_k$ is an 
  \dfn{$S$-word} for $w$. The \dfn{length} of $w$, denoted $\ell(w)$, is 
  the minimum length of an $S$-word for $w$, and any $S$-word of this 
  length is \dfn{reduced}. It is known that $u\leq v$ holds in the weak 
  order if and only if $\ell(u)+\ell(u^{-1}v) = \ell(v)$.
\end{definition}

Dyer's extension of the weak order relies on a geometric characterization 
of inversion sets, which we recall next.

\begin{definition}[Separable sets]
  Let $\mathcal{V}$ be a set of nonzero vectors in $\RR^n$, and let $X$ 
  be a subset of $\mathcal{V}$. All separating hyperplanes in this definition 
  are linear. In dimension zero, the empty set is weakly separable.
  \begin{itemize}
    \item The set $X$ is \dfn{separable} if there exists a hyperplane $H$ 
    such that $X$ and $\mathcal{V}\setminus X$ are contained in different 
    open half-spaces defined by $H$.
    \item The set $X$ is \dfn{weakly separable} if there exists a hyperplane 
    $H$ such that $X$ and $\mathcal{V}\setminus X$ are contained in different 
    closed half-spaces defined by $H$, and $X\cap H$ is a weakly separable 
    subset of $\mathcal{V}\cap H$ as a set of vectors of $H\cong \RR^{n-1}$.
    \item The set $X$ is \dfn{biclosed} if both $X$ and $\mathcal{V}\setminus 
    X$ are closed under taking vectors of $\mathcal{V}$ in the nonnegative 
    span of any two of their elements.
  \end{itemize} 
\end{definition}

It follows from the definitions that any separable set is weakly separable, 
and any weakly separable set is biclosed. For $\mathcal{V}=\Phi^+$, 
biclosedness is equivalent to requiring that $X\cap P$ be separable in 
$\Phi^+\cap P$ for every two-dimensional linear subspace $P$. In this case,
there is a well-known partial converse \cite[Proposition 2.11]{HohlwegLabbe}.

\begin{lemma}[Characterization of inversion sets]\label{lem:inversion-sets}
  Let $X$ be a finite subset of $\Phi^+$. The following statements are equivalent.
  \begin{enumerate}
    \item $X$ is biclosed.
    \item $X$ is separable.
    \item $X$ is the inversion set of some $w\in W$.
  \end{enumerate}
\end{lemma}

\begin{corollary}\label{cor:finitecofinite}
  If $X$ is a biclosed subset of $\Phi^+$ that is either finite or cofinite, 
  then $X$ is separable.
\end{corollary}

\begin{definition}[Extended weak order]
  The \dfn{extended weak order} for $(W,S)$ is the poset of biclosed subsets 
  of $\Phi^+$, ordered by inclusion.
\end{definition}

\begin{example}[Rank $2$]\label{ex:rank-2}
  If $W$ is a rank 2 (i.e. dihedral) Coxeter group, then as in \Cref{rem:rank-2}, 
  there are three types of root system for $W$: finite, affine, or hyperbolic.
    
  If $\alpha,\beta$ are the simple roots in this rank $2$ subsystem with 
  corresponding reflections $r=r_\alpha$ and $s=r_\beta$, then the three cases 
  are determined by the pairing $(\alpha,\beta)$. In all three cases, there is 
  a unique total order of the positive roots putting
  \[ 
  \alpha \prec r\cdot \beta \prec rs\cdot \alpha \prec \cdots 
  \prec sr\cdot \beta \prec s\cdot \alpha \prec \beta. 
  \]
  The biclosed sets are exactly the initial and terminal sections of this ordering. 
\end{example}

\begin{remark}[Independence of root system] \label{independence2}
  For a given Coxeter system $(W,S)$ there may be many choices for a root system 
  $\Phi$ whose corresponding reflection group is isomorphic to $W$ (\cref{rem:finite-area}).
  Using the bijection $\Phi^+\leftrightarrow T$, one can view a subset of positive 
  roots as a subset of reflections. As a subset of $T$, a biclosed set can be 
  characterized directly in terms of the Coxeter system, and so the notion of 
  biclosed is independent of the choice of root system $\Phi$; see 
  \cite[Section 11.2]{dyer:2019weak}. On the other hand, in rank $4$ or greater, 
  the weakly separable sets may depend on the choice of root system. In rank $3$, 
  the weakly separable sets  are independent of the choice of $\Phi$; see \cite{DyerWang2021}.
\end{remark}

Matthew Dyer made the following two conjectures about extended weak order, 
which we resolve for rank~$3$ Coxeter groups in this paper.

\begin{conjecture}[Conjecture A]
  If $\cC$ is a maximal chain in extended weak order, there is a total ordering 
  $\prec$ of $\Phi^+$ such that each initial section of $\prec$ is in $\cC$. 
  Equivalently, for any cover relation $B_1\lessdot B_2$ in extended weak order, 
  we have $|B_2\setminus B_1|=1$.
\end{conjecture}

Recall that a complete lattice is a poset for which any collection of elements 
has a least upper bound and greatest lower bound. Dyer's second conjecture is 
the following.

\begin{conjecture}[Conjecture B/Lattice Conjecture]
  Extended weak order is a complete lattice.
\end{conjecture}

For root systems, we recall a weaker form of biclosed sets introduced in 
\cite{BDHMMS25}.

\begin{definition}[Parabolic biclosed]\label{dfn:parabolic_biclosed}
  A subset $X$ of $\Phi^+$ is \dfn{parabolic biclosed} if for every $2$-dimensional 
  subspace $P\subseteq \RR^n$ that supports a rank $2$ parabolic subsystem of $\Phi$, 
  the set $X\cap P$ is a separable subset of $\Phi^+\cap P$. 
\end{definition}

\begin{lemma}
  For any rank $2$ parabolic subsystem $\Phi'$, we have $\Phi' = \Phi\cap \Span(\Phi')$.
\end{lemma}

\begin{corollary}
  A subset $X$ of $\Phi^+$ is parabolic biclosed if and only if for any rank~$2$ 
  parabolic subsystem $\Phi'$, the set $X\cap (\Phi')^+$ is a separable subset of 
  $(\Phi')^+$.
\end{corollary}

We define an action of $W$ on subsets of $\wh{\Phi}$ as follows.
Given sets $X,Y$, let $X\triangle Y = (X\cup Y)\setminus(X\cap Y)$ be the 
\dfn{symmetric difference} of $X$ and $Y$. For $w\in W$ and $X\subseteq\wh{\Phi}$, 
let \[ w\cdot X = \wh{\Inv}(w)\triangle\{\wh{w\cdot\alpha}:\ \wh\alpha\in X\}. \]
This also defines an action of $W$ on subsets of $\Phi^+$. If $X\subseteq\Phi^+$, 
we let $w\cdot X = \{\alpha\in\Phi^+:\ \wh{\alpha}\in w\cdot\wh{X}\}$; see 
\cite[Lemma 4.1(i)]{dyer:2019weak}.

\begin{lemma}\label{lem:action}
  Let $w\in W,\ X\subseteq\Phi^+$.
  \begin{enumerate}
  \item If $X$ is weakly separable, then $w\cdot X$ is weakly separable.
  \item If $X$ is biclosed, then $w\cdot X$ is biclosed.
  \item If $X$ is parabolic biclosed, then $w\cdot X$ is parabolic biclosed.
  \end{enumerate}
\end{lemma}

\section{Additional results on rank 3 Coxeter groups}
In this section we record a few results about rank $3$ hyperbolic 
Coxeter groups that are elementary, but for which we can find no reference 
in the literature.  Our first result is a characterization of when
the line through two rescaled roots can be tangent to the disk and 
the second discusses the corresponding positive half-space.

\begin{lemma}[Parabolic Cusps]\label{lem:cusps}
  If $\widehat{\alpha}, \widehat{\beta} \in \widehat{\Phi}$ are rescaled 
  roots of the Coxeter group $\Cox(p,q,r)$ and the line 
  $L=L(\widehat{\alpha},\widehat{\beta})$ is tangent to $\partial \DD$, 
  then the point $x$ of tangency is the ideal point of an infinite face 
  $F = w \langle S' \rangle$, and the rescaled roots $\widehat\Phi\cap L$ 
  are the $w$ translates of the roots in the standard rank $2$ parabolic 
  subgroup $\langle S' \rangle$.  In particular, $r=\infty$.
\end{lemma}

\begin{figure}
  \centering
  \includegraphics[height=60.105mm]{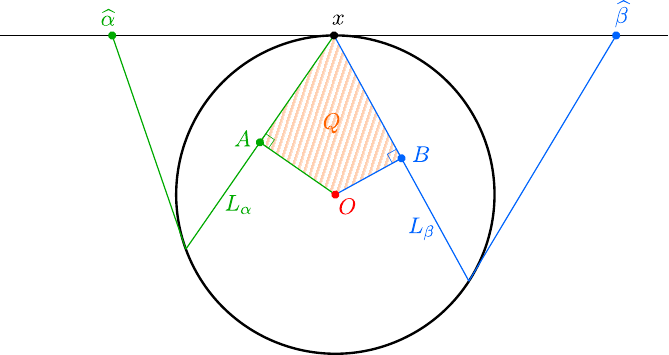}
  \caption{Diagram for the proof of \cref{lem:cusps}}
  \label{fig:cusps}
\end{figure}

\begin{proof}
  The lines $L_\alpha$ and $L_\beta$ meet at the point $x \in \partial \DD$ 
  and the reflections $r_\alpha$ and $r_\beta$ generate an infinite dihedral 
  group $G = \langle r_\alpha, r_\beta \rangle$ fixing $x$.  The roots in $G$
  approach $x$ from both sides (\cref{rem:face-lines}).  Thus, without loss 
  of generality, we may assume that the point $x$ is between $\widehat\alpha$ and 
  $\widehat \beta$.  In $\DD$ this corresponds to $O$ being between the lines
  $L_\alpha$ and $L_\beta$.
  Let $A$ and $B$ be the points on $L_\alpha$ and $L_\beta$ closest to $O$, and 
  consider the convex quadrilateral $Q$ with vertices $A$, $O$, $B$ and ideal 
  point $x$. See \cref{fig:cusps}. The area of $Q$ is finite.\footnote{Bounded 
  above by $\pi$ since its area is $2\pi$ minus the sum of its angles and the 
  angles at $A$ and $B$ are right angles.}  Let $T$ be the union of chambers in 
  the hyperbolic triangular tiling that contain a point in the interior of $Q$.  
  Since these triangles have a constant non-zero area there are only finitely 
  many completely inside $Q$.\footnote{The group $\Cox(2,3,7)$ has the smallest 
  triangle with area $\frac{\pi}{42}$, so there are at most $42$ chambers completely 
  in $Q$.}
  There are also only finitely many that are partially inside $Q$.  Any such 
  triangle must have a side that intersects the line segments $OA$ or $OB$; 
  they cannot cross $L_\alpha$ and $L_\beta$ since both lines are part of the 
  $1$-skeleton of the triangular tiling.  These sides extend to lines passing 
  within a bounded distance of $O$, which means the corresponding roots 
  are bounded away from $\partial \DD$.  This set of potential roots is finite 
  by \cref{accumulationpoints} since it is contained in a compact subset of 
  $\AA \setminus \DD$, inside the triangle $\conv(\widehat\alpha_1, \widehat\alpha_2, 
  \widehat\alpha_3)$ but bounded away from $\DD$.  Finally, each side crossing
  $OA \cup OB$ is a side of at most $2$ chambers. 
  
  If $p \leq q \leq r < \infty$, each chamber has finite diameter, and we get 
  a contradiction.  There are points in $Q$ that are arbitrarily far from $O$ 
  (particularly those approaching $x$), but the points in the chambers with points 
  in $Q$ are bounded distance from $O$ (finitely many chambers plus finite diameter).  
  Thus, $r=\infty$.  The finitely many chambers intersecting $Q$ have finitely 
  many centers $\{w_i\cdot O\}$ that are a bounded distance from $O$.  Any 
  point $C$ in $Q$ that is sufficiently far from $O$ must be covered by
  a chamber whose center $w\cdot O$ is relatively close to $O$ and far from $C$.  
  The point $C$ in a chamber arbitrarily far from its center must be in a corner 
  of the chamber with an ideal vertex $y$.  And given the fact that the two sides 
  of the chamber meeting at $y$ (labeled, say $s_i$ and $s_j$) cannot cross the 
  lines $L_\alpha$ and $L_\beta$,  the ideal vertex $y$ must be at $x\in \partial \DD$.  
  This makes the line through $\widehat \alpha$ and $\widehat \beta$ the line for 
  the reflections in the conjugate subgroup $w \cdot \langle s_i, s_j\rangle\cdot w^{-1}$.
\end{proof}

The contra-positive of \Cref{lem:cusps} says that when $p$, $q$, and $r$ are finite, 
the lines determined by the rescaled roots $\widehat \Phi$ of $\Cox(p,q,r)$ are never 
tangent to $\partial \DD$.

\begin{proposition}\label{prop:finite_visibility}
  Fix $x\in\partial\DD$.  If the tangent line  $L_x$ contains two rescaled roots 
  of $\widehat\Phi$ then the intersection $\resroots \cap L_x^+$ is finite.  Conversely, 
  if the intersection $\resroots \cap L_x^+$ is finite, then $L_x$ contains 
  the rescaled roots of a parabolic rank $2$ subsystem. 
\end{proposition}

\begin{proof} 
  If $L_x$ contains at least two rescaled roots, then 
  by \cref{lem:cusps}, $L_x$ contains the rescaled roots of a rank $2$ 
  parabolic root system $\Phi'$. Let $W_J=\langle s_i,s_j\rangle$ and 
  let $w$ be the minimal-length representative of the parabolic coset $wW_J$ 
  corresponding to this infinite face. Write $\Phi_J$ for the root system of $W_J$, 
  so that $\Phi'=w\Phi_J$. Geometrically, $\widehat{w\cdot O}$ is the unique 
  source on the boundary of the infinite face $F=\conv\{\widehat{g\cdot O}:g\in wW_J\}$.
  Then $w^{-1}\Phi'=\Phi_J$ is a standard parabolic subsystem of $\Phi$ and $w^{-1}\cdot L_x = 
  L(\widehat{\alpha}_i,\widehat{\alpha}_j)$ is a line spanned by 
  simple roots tangent to $\partial\DD$ at the point $y=w^{-1}\cdot x$ 
  (\Cref{rem:inversion}).  Note that $L_y^+ \cap \widehat\Phi = \emptyset$ 
  as there are no roots outside the triangle 
  $\conv(\widehat{\alpha}_1,\widehat{\alpha}_2,\widehat{\alpha}_3)$. 
  The roots in $\widehat\Phi\cap L_x^+$ correspond precisely to the 
  reflecting hyperplanes separating $O$ from the face $F$. Since $w$ is the 
  minimal-length representative of $wW_J$, these roots form $\widehat{\Inv(w)}$, 
  which is finite.
     
  In the other direction, suppose $L_x^+$ contains finitely many rescaled 
  positive roots $B = \{\gamma\in\Phi^+:\ \langle x,\gamma\rangle> 0\}$.  
  Then $B$ is a finite separable set: adding $\varepsilon\langle O,\gamma\rangle$ 
  to $\langle x,\gamma\rangle$ for sufficiently small $\varepsilon>0$ preserves 
  positivity on $B$ and makes the value negative on every other positive root. 
  Thus, $B = \Inv(w)$ for some $w\in W$ (\Cref{lem:inversion-sets}). 
  The line $w^{-1}\cdot L_x$ is tangent to $\partial\DD$ at $y=w^{-1}\cdot x$ and
  by \Cref{rem:inversion}, $\resroots\cap L_y^+ = \emptyset$. Since $\DD$ is 
  contained in the triangle $\conv(\widehat{\alpha}_1,\widehat{\alpha}_2,\widehat{\alpha}_3)$, 
  we conclude that $L_y$ contains a side of this triangle, i.e. $L_y= 
  L(\widehat{\alpha}_i,\widehat{\alpha}_j)$. Therefore, $L_x$ contains the rescaled 
  roots of the rank $2$ parabolic subsystem $w\Phi_J$, where $\Phi_J$ is the 
  root system of $\langle s_i,s_j\rangle$.
\end{proof}

\section{Root Polygons}\label{sec:rootpolygon}

In this section, we introduce root polygons and establish their 
basic properties.

\begin{definition}[Root polygon]\label{root-polygon-def}
  A \dfn{root polygon} $P$ is a convex polygon in $\AA$ enclosing 
  $\DD$ whose vertices are (rescaled) roots. Thus $P = \conv(X) 
  \subset \AA$ with $X\subset \resroots$, $\abs{X}<\infty$, and 
  $\DD \subset \conv(X)$.
\end{definition}

Note that in the definition of a root polygon $\pgon$, it is possible 
that $\DD$ intersects $\partial \pgon$.  

\begin{lemma}\label{lem:rootpolygoncomplement}
  Let $\pgon$ be a root polygon. Then there are finitely many rescaled 
  roots in $\AA\setminus\pgon$. 
\end{lemma}

\begin{proof}
  Suppose $L$ is a line supporting an edge of $\pgon$, and let $L^+$ be 
  the open half space defined by $L$ that does not contain $\pgon$. If $L$ 
  is not tangent to $\partial\DD$, then $L^+\cap\wh{\Phi}$ is finite by 
  \cref{accumulationpoints}. If $L$ is tangent to $\partial\DD$, then 
  $L^+\cap\wh{\Phi}$ is finite by \cref{prop:finite_visibility}. The complement 
  of $\pgon$ is the union of the half spaces $L^+$ as $L$ runs over the lines 
  bounding $\pgon$.  Since $\pgon$ has finitely many edges, it follows that 
  $(\AA\setminus \pgon)\cap\wh{\Phi}$ is the union of a finite number of finite 
  sets.    
\end{proof}

\begin{remark}
  If $X$ is an arbitrary finite subset of $\AA$ such that $\DD\subseteq \conv(X)$, 
  then the statement of \Cref{lem:rootpolygoncomplement} can fail: there may 
  be infinitely many rescaled roots in the complement of $\conv(X)$. It is 
  important that the vertices of a root polygon are rescaled roots, because this 
  implies that the edges of the polygon that are tangent to $\DD$ are in fact 
  parabolic subsystems.
\end{remark}

\begin{definition}[Polar Dual Polygons]
  Let $\pgon$ be a convex subset of $\AA$. The \dfn{polar dual} of $\pgon$ is 
  the set
  \[ 
  \pgon^\diamond \coloneqq \{ x \in \AA \mid \langle x,p\rangle \leq 0 
  \text{ for all $p\in \pgon$} \}. 
  \]
\end{definition}

Let $C_0$ be the identity chamber. When $\pgon$ is a convex polygon and 
$\DD \subseteq \pgon \subseteq \conv(\widehat\alpha_1,\widehat\alpha_2,\widehat\alpha_3)$ 
(respectively, $C_0 \subseteq \pgon\subseteq \DD$), then $\polar$ is a convex polygon 
with $C_0 \subseteq \polar\subseteq \DD$ (respectively, $\DD \subseteq \polar 
\subseteq \conv(\widehat\alpha_1,\widehat\alpha_2,\widehat\alpha_3)$). Moreover, 
in these cases, $\pgon^{\diamond\diamond}=\pgon$. Hence, in these cases we may discuss 
the edges and corners of $\polar$. 

\begin{lemma}[Polar Dual Polygons]\label{lem:polar}
  Let $\pgon$ be a root polygon. Then $\polar$ is contained in $\DD$ and 
  contains the identity chamber. A rescaled root $\widehat\alpha$ is a vertex 
  of $\pgon$ if and only if $\widehat H_\alpha$ is a supporting hyperplane for 
  an edge of $\polar$. A rescaled root $\widehat\alpha$ is in $\partial \pgon$ 
  if and only if $\widehat H_\alpha$ intersects $\polar$ at a corner.
\end{lemma} 

Let $\mathrm{int}(\polar)$ denote the interior of $\polar$. 

\begin{lemma}\label{cuts-vs-outside}
  Let $\pgon$ be a root polygon.  For every $\alpha\in \Phi^+$, we have 
  $\widehat\alpha\not\in \pgon$ if and only if $\widehat H_\alpha\cap 
  \mathrm{int}(\polar)\neq\emptyset$. 
\end{lemma}

\begin{proof}
  Since $\pgon=\pgon^{\diamond\diamond}$, we know that $\widehat\alpha\in \pgon$ 
  if and only if $\langle x,\widehat\alpha\rangle\leq 0$ for every $x\in \polar$. 

  Suppose first that $\widehat\alpha\in \pgon$.  Then the affine-linear function
  $x\mapsto \langle x,\widehat\alpha\rangle$ is nonpositive on $\polar$.  In fact,
  it is strictly negative on $\mathrm{int}(\polar)$. Hence, $\widehat H_\alpha$ 
  does not meet $\mathrm{int}(\polar)$. 

  Conversely, suppose that $\widehat\alpha\notin \pgon$.  Then there is some 
  point $x\in \polar$ such that $\langle x,\widehat\alpha\rangle>0$.  
  Choose a point $x_0$ in the interior of the identity chamber, which lies in 
  $\mathrm{int}(\polar)$ by \cref{lem:polar}.  Since $\alpha$ is a positive root,
  we have $\langle x_0,\widehat\alpha\rangle<0$.  The line segment from $x_0$ to
  $x$ lies in $\polar$, and the function $y\mapsto \langle y,\widehat\alpha\rangle$ 
  changes sign along this line segment. Therefore, the line segment meets 
  $\widehat H_\alpha$.  The intersection point lies in $\mathrm{int}(\polar)$ 
  because every point of the segment except possibly $x$ is in the interior 
  of $\polar$.  Thus, $\widehat H_\alpha\cap \mathrm{int}(\polar)\neq\emptyset$.
\end{proof}

\begin{corollary}\label{finite-regions}
  Let $\pgon$ be a root polygon.  Only finitely many hyperplanes
  $\widehat H_\alpha$ with $\alpha\in\Phi^+$ intersect
  $\mathrm{int}(\polar)$.  Consequently, the restriction of the
  Coxeter arrangement to $\polar$ has only finitely many regions meeting
  $\mathrm{int}(\polar)$.  Equivalently, only
  finitely many faces of $\mathbb D\setminus \Gamma$ intersect
  $\mathrm{int}(\polar)$.
\end{corollary}

\begin{proof}
  By \cref{cuts-vs-outside}, every hyperplane that intersects
  $\mathrm{int}(\polar)$ is indexed by a positive root
  $\alpha$ such that $\widehat\alpha\notin \pgon$. 
  By \Cref{lem:rootpolygoncomplement}, a root polygon contains all
  but finitely many rescaled positive roots, so there are only finitely many
  such $\alpha$.  Hence, only finitely many reflecting hyperplanes meet
  $\mathrm{int}(\polar)$.
\end{proof}

\section{The Polygon Lemma} \label{sec:polygonlemma}

In this section we will prove an important lemma about the structure of 
parabolic-biclosed sets, namely that their restriction to the boundary of 
any root polygon $\pgon$ is separable. We first reframe this result in terms 
of colors of roots on $\partial\pgon$. Fix a decomposition $\Phi^+ = B\sqcup R$ 
of the positive roots into a parabolic biclosed set and its complement. 
We refer to the roots in $B$ as \emph{blue} and the roots in $R$ as \emph{red}. 
We similarly attribute colors to the rescaled roots $\resroots$. 

Given a root polygon $\pgon$, we say that rescaled roots 
$\widehat\alpha,\widehat\beta\in \partial \pgon$ are \dfn{consecutive} in 
$\partial\pgon$ if the line segment between $\widehat\alpha$ and $\widehat\beta$ 
is contained in $\partial\pgon$ and contains no other root.
A \dfn{color change} in $\partial \pgon$ is either a pair of consecutive 
roots in $\partial \pgon$ that are different colors or a point in $\partial \pgon$ 
that is the limit point of both red rescaled roots in $\partial \pgon$ and blue
rescaled roots in $\partial \pgon$. Since each edge of $\pgon$ spans a line that 
supports a rank $2$ parabolic subsystem of $\Phi$, there is at most one color change 
along each edge of $\pgon$. We now state the main result of this section. 

\begin{lemma}[Polygon Lemma]\label{lem:polygon-lemma}
  For any root polygon $\pgon$, there are at most two color 
  changes in $\partial \pgon$. 
\end{lemma}

If there are two color changes in $\partial \pgon$, then a line 
through these color changes separates $B\cap\partial \pgon$ and 
$R\cap\partial \pgon$.  Otherwise, by \cref{lem:polygon-lemma}, 
the boundary of $\pgon$ is monochromatic.  We will see that this 
case only occurs when $B$ or $R$ is finite.

Our proof of \Cref{lem:polygon-lemma} will rely heavily on certain 
features of the Cayley graph $\Gamma$. We transfer the $2$-coloring 
of the roots to a coloring of the edges of $\Gamma$ by giving the edge
$e$ the same color as $\alpha_e$. As we did for root polygons, we define
a \dfn{color change} in $\Gamma$ to be a pair of edges with different 
colors that share a common vertex $v$ and face $F$.  We say the color 
change \dfn{occurs at $v$ in $F$}.

\begin{observation}[Opposite Edges]\label{opposite-edges}
  Let $F$ be a face of $\Gamma$.  When $F$ is an infinite face, the edges 
  in $\partial F$ correspond to distinct roots.  When $F$ is a finite face
  of $\Gamma$, it has an even number of edges (and it is regular in the 
  intrinsic hyperbolic metric).  The hyperplane that passes through any edge 
  $e$ also passes through the center of $F$ and through the opposite edge $e'$ 
  in the polygon $F$.  In particular, $e$ and $e'$ determine the same root 
  $\alpha_e=\alpha_{e'}$ and thus opposite edges of $F$ always have the same color.
\end{observation}

\begin{figure}
  \centering
  \includegraphics[height=4.012cm]{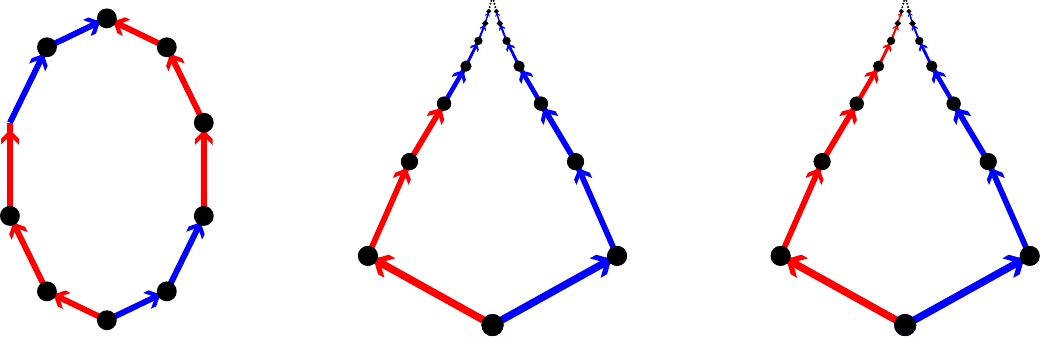}
  \caption{The different color patterns on a bichromatic face.}
  \label{fig:facecolors}
\end{figure}

If $F$ is a face of the Cayley graph whose edges all have the same color, 
then we say that $F$ is a \dfn{monochromatic face}; otherwise $F$ is a 
\dfn{bichromatic face}. Combining \Cref{opposite-edges,rem:face-lines}
leads to the following observation, illustrated in \Cref{fig:facecolors}.

\begin{observation}
  Let $F$ be a bichromatic face of the Cayley graph.
  
  If $F$ is finite, then there must be a color change at both the source 
  vertex of $F$ and the sink vertex of $F$.  Moreover, on each of the two 
  paths from the source of $F$ to the sink of $F$, there must be exactly 
  one color change. In particular, $F$ has exactly four vertices with color 
  changes. 
  
  If $F$ is instead an infinite bichromatic face, then there must be a color 
  change at the source vertex of $F$. Moreover, exactly one of the following 
  two possibilities occurs:
  \begin{itemize}
    \item Among the two unbounded paths from the source vertex of $F$, one 
    has a (unique) color change and the other does not.
    \item Among the two unbounded paths from the source vertex of $F$, one 
    is entirely red and the other is entirely blue.
  \end{itemize}
\end{observation}

\begin{lemma}\label{lem:identity}
  If $R$ and $B$ are both nonempty, then there is a color change at $v_{\id}$.
\end{lemma}

\begin{proof}
  Suppose $R$ and $B$ are both nonempty.
  Since $\Gamma$ is connected, there is a vertex with a color change. 
  Choose $w\in W$ of minimum length such that there is a color change at $v_w$.
  If $w$ is not the identity, choose a descent edge at $v_w$. At least one of the 
  two faces containing this edge has a color change at $v_w$; choose such a face $F$. 
  Then $v_w$ is not the source of $F$. But then there is a color change at 
  $v_{\yy(F)}$, contradicting the minimality of the length of $w$.
\end{proof}

To use these preliminary observations about the coloring of $\Gamma$, we 
transfer the vertex coloring of the root polygon $\pgon$ into an edge coloring 
of its polar dual $\polar$, which is inside $\DD$ by \cref{lem:polar}.
Our plan for the proof of the Polygon Lemma (\cref{lem:polygon-lemma}) 
is to construct objects called \emph{snakes} that will separate the red 
Cayley graph edges from the blue ones, doing so in such a way that either 
zero or two snakes cross the boundary of $\polar$.  This will imply that 
there are either zero or two color changes in the boundary of $\polar$, 
from which the Polygon Lemma will follow.  A similar argument was used in 
\cite{BDHMMS25} to deduce cleanliness for the rank 3 universal Coxeter group. 

\begin{figure}
  \centering
  \includegraphics[height=4.012cm]{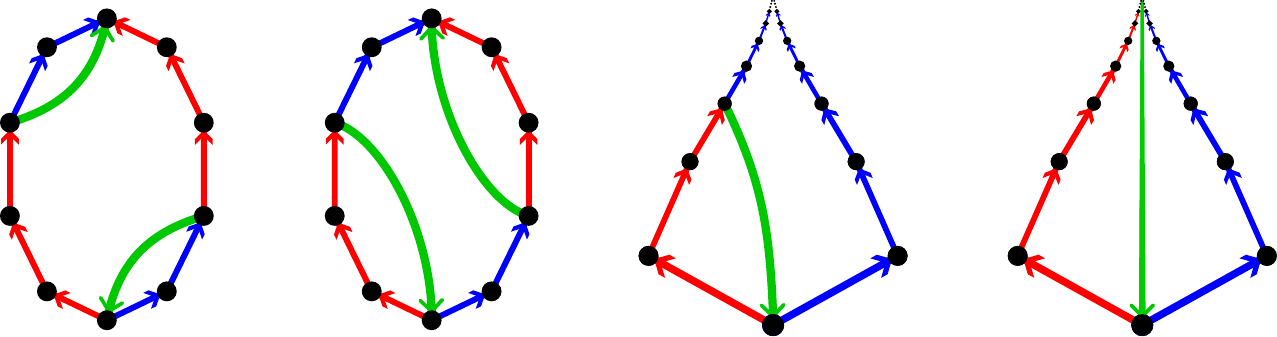}
  \caption{The possible snake segments in a bichromatic face. 
  A finite bichromatic face has two possible snake segment configurations, 
  while an infinite bichromatic face has only one.}
  \label{fig:snakesegments}
\end{figure}

We build snakes from certain components called \emph{snake segments}. The 
rules for building these pieces are illustrated in \Cref{fig:snakesegments}.
All together, these segments form a 1-manifold embedded in $\DD$ called a 
\emph{snake nest}, which decomposes into oriented components called \emph{snakes}.

\begin{definition}[Snake segment]\label{def:snakeseg}
  Let $V_{\cS}$ be the set of vertices of $\Gamma$ at which there 
  is a color change. We also include in $V_{\cS}$ any ideal points 
  that are part of an infinite bichromatic face with exactly one color 
  change.  A \dfn{snake segment} supported by a face $F$ of $\Gamma$ 
  is an oriented curve embedded in $F$ whose endpoints are color-change 
  vertices in $F$, or the ideal point of $F$ when $F$ is infinite with 
  exactly one color change, and whose interior is contained in the 
  interior of $F$.  We require that one endpoint of a snake segment is 
  either the source or sink of $F$, if it exists, and the other endpoint 
  is neither the source nor sink. A snake segment is always oriented 
  toward the source or sink of $F$.
\end{definition}

\begin{definition}[Snake nest]
  A \dfn{snake nest} is a directed graph $(V_{\cS},E_{\cS})$ embedded in 
  $\DD$ where $E_{\cS}$ is a maximal family of snake segments with the 
  property that any two snake segments supported by the same face do not 
  intersect, even at endpoints.
\end{definition}

Within a finite bichromatic face, there are two non-isotypic choices for 
the pair of snake segments supported by the face; see the left two diagrams 
in \Cref{fig:snakesegments}. In an infinite bichromatic face, there is a 
unique choice; see the right two diagrams in \Cref{fig:snakesegments}.
Typically, there are many possible snake nests one may associate to a given 
coloring of the edges of $\Gamma$.  We will use this flexibility later.

\begin{lemma}[Consistency of orientation]\label{lem:orientation}
  Assume both $R$ and $B$ are nonempty.
  Fix a snake nest $(V_{\cS},E_{\cS})$, and let $v\in V_{\cS}$ such that 
  $v$ is not an ideal point. Let $w\in W$ such that $v=v_w$.
  \begin{enumerate}
     \item There are precisely two snake segments incident to $v$.
     \item Both snake segments incident to $v_{\id}$ are oriented toward $v_{\id}$.
     \item If $R=\Inv(w)$ or $B=\Inv(w)$, then both snake segments incident to 
     $v_w$ are oriented away from $v_w$.
     \item If $w\neq\id,\ R\neq\Inv(w),$ and $B\neq\Inv(w)$, then one snake 
     segment incident to $v_w$ is oriented toward $v_w$ and the other snake 
     segment is oriented away from $v_w$.
  \end{enumerate}
\end{lemma}

\begin{proof}
  (1) An edge $e$ of $\Gamma$ incident to $v$ has a different color from the 
  other two incident edges. Consequently, there is a color change at $v$ in 
  the two faces containing $e$. By the maximality of the snake nest, each of 
  those faces supports a snake segment with endpoint $v$.
  
  (2) The vertex $v_{\id}$ is the source of each of its incident faces, so 
  its snake segments are oriented towards $v_{\id}$. 

  (3) For the two remaining parts, we may assume $w\neq \id$.
  Either $w$ has two descents and one ascent, or $w$ has one descent and two ascents.
  In the former case, $v_w$ is the sink of one of its incident faces and is 
  neither source nor sink of the other two faces. In the latter case, 
  $v_w$ is the source of one of its incident faces and is neither source nor 
  sink of the other two faces. In either case, we see that there is at most 
  one snake segment incident to $v_w$ that is oriented toward $v_w$.
  
  Now assume $R=\Inv(w)$.
  If $w$ has two descents, then the face of $\Gamma$ in which $v_w$ is the sink is all red.
  If $w$ has two ascents, then the face in which $v_w$ is the source is all blue.
  In either case, the color changes at $v_w$ occur in the two faces in which $v_w$ 
  is neither source nor sink. Hence, both snake segments are oriented away from $v_w$.
  The case $B=\Inv(w)$ is similar.

  (4) Finally, we assume $w\neq\id$ and both snake segments are oriented 
  away from $v_w$. By symmetry, we may assume the roots labeling descents 
  of $w$ are red and roots labeling ascents of $w$ are blue.
  By applying $w^{-1}$ to the coloring of $\Gamma$, all edges incident to 
  $v_{\id}$ become blue. But $w^{-1}\cdot R$ is a parabolic biclosed set 
  by \cref{lem:action}, so by \cref{lem:identity}, $w^{-1}\cdot R=\emptyset$.
  Therefore, $R=w\cdot\emptyset = \Inv(w)$.
\end{proof}

\begin{corollary}[Manifold]
  A snake nest is a $1$-manifold embedded in $\DD$, possibly with boundary 
  on $\partial\DD$.
\end{corollary}

\begin{proof}
  Within a given face $F$ of $\Gamma$, the restriction of a snake nest to $F$ 
  is either empty or a $1$-manifold with boundary on the boundary of $F$.
  Using \cref{lem:orientation}(1), the segments glue together to give a 
  $1$-manifold embedded in $\DD$ that either has no boundary if there are no faces 
  with exactly one color change, or it has boundary contained in $\partial\DD$. 
\end{proof}

\begin{definition}[Snake]
  A \dfn{snake} is a maximal oriented path or cycle of a given snake nest.
  For any snake $\cS$ and points $u,v\in\cS$, if $\cS$ is oriented from $u$ 
  to $v$, then we say $\cS$ \dfn{travels} from $u$ to $v$.
  In particular, if $\cS$ is an oriented cycle, then $\cS$ travels from $u$ 
  to $v$ for any two of its points $u,v$.
\end{definition}

Using \Cref{lem:orientation}, we may determine all possible endpoints of snakes.

\begin{corollary}[Endpoints of snakes]\label{cor:endpoints}
  If $v$ is an endpoint of a snake, then either $v=v_{\id}$, $v$ is an ideal point, 
  or $v=v_w$ for some $w\in W$ such that either $R=\Inv(w)$ or $B=\Inv(w)$.
\end{corollary}

\begin{remark}[Comparison to snakes from previous work]
  We introduced snakes in \cite{BDHMMS25} in the case where $(W,S)$ is the 
  universal Coxeter system of rank $3$. In that case, we referred to a snake 
  nest as a pair of snakes since there were exactly two snakes for any given
  parabolic biclosed set, each of which had one endpoint at $v_{\id}$.
  In the more general setting of this paper, a snake nest may contain many snakes.
  A snake may be an oriented cycle or a path with $0$, $1$, or $2$ endpoints 
  as described by \cref{cor:endpoints}.
\end{remark}

\begin{definition}[Compatibility]
  Let $\pgon$ be a root polygon. A snake nest is \dfn{compatible} with $\pgon$ 
  if all intersection points of the snake nest with $\partial \polar$ are corners 
  of $\polar$. 
\end{definition}

\begin{figure}
  \centering
  \includegraphics[height=3.869cm]{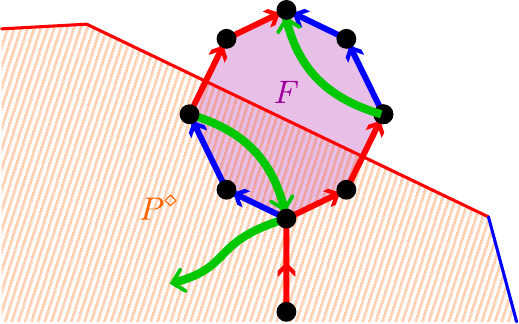}\qquad\qquad
  \includegraphics[height=5.5cm]{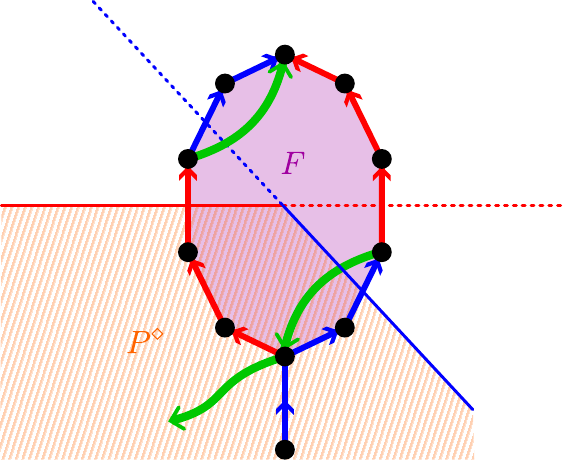}
  \caption{On the left is a face $F$ that intersects the polygon $\polar$ so that 
  $F\cap\partial\polar$ is a line segment. On the right is a face $F$ that 
  intersects the polygon $\polar$ so that $F\cap\partial\polar$ is a union of two 
  line segments.} \label{fig:proofcompatiblesnakes1}
\end{figure}

\begin{lemma}[Compatible Snakes]\label{lem:compatiblesnakes}
  For any root polygon $\pgon$, there is a choice of snake nest that is compatible 
  with $\pgon$.
\end{lemma}

\begin{proof}
  Let $\pgon$ be a root polygon, and let $V_{\cS}$ be the set of vertices defined 
  in \cref{def:snakeseg}. For the purpose of constructing a snake nest 
  compatible with $\pgon$, we choose snake segments supported by each face $F$ 
  of $\Gamma$ as follows. We first dispense with two trivial cases.
  If $F$ is monochromatic, then it does not support any snake segments.
  If $F$ does not meet $\partial \polar$, we may arbitrarily pick a maximal 
  collection of nonintersecting snake segments in that face.
    
  Hence, we assume $F$ is bichromatic and $F\cap \partial\polar$ is nonempty.
  There are two possible shapes for this intersection, illustrated in 
  \Cref{fig:proofcompatiblesnakes1,fig:proofcompatiblesnakes2}: either 
  $F\cap \partial \polar$ is a line segment, or $F\cap \partial \polar$ is 
  the union of two line segments, meeting at a corner of $\polar$. 
  In both cases, there are two points in $\partial F\cap \partial \polar$ 
  and there are two components of $F\setminus \partial \polar$. If $F$ is 
  finite, then the source and sink are in different components.

  Let $C$ be the set of vertices at which a color change occurs in $F$, 
  together with the ideal point of $F$ if $F$ is infinite with exactly one 
  color change. The size of $C$ is $2$ or $4$ depending on whether $F$ is 
  infinite or finite, respectively. To determine how to link points of $C$ 
  with snake segments, we separately consider two cases based on the colors 
  of the two points in $\partial F\cap \partial \polar$.

  \noindent\textbf{Case 1.} The two points in $\partial F\cap \partial \polar$ 
  have the same color.
    
  In this case, there is an even number of color changes in $\partial F$ on 
  either side of $\partial \polar$. In other words, $C$ is partitioned into 
  two sets of even size. If $F$ is infinite, then $|C|=2$, and we may select 
  a snake segment supported by $F$ that does not cross $\partial\polar$.

  If $F$ is finite, then $|C|=4$ and the source and sink vertices are on 
  different sides of $\partial\polar$.  Since $C$ is partitioned into two 
  sets of even size by $\partial \polar$, it follows that there are a pair 
  of snake segments supported by $F$ that do not cross $\partial\polar$.

  \noindent\textbf{Case 2.} The two points in $\partial F\cap \partial \polar$ 
  have different colors.

  In this case, $F\cap \partial \polar$ must be the union of two line segments 
  meeting at a corner of $\polar$. Moreover, there is an odd number of color 
  changes in $\partial F$ on either side of $\partial \polar$. Hence, $C$ is 
  partitioned into two sets of odd size. Thus, for any choice of 
  nonintersecting snake segments, exactly one segment has endpoints on 
  opposite sides of $\partial\polar$.  We may deform this segment so that 
  it passes through the corner of $\polar$ in $F$. 
\end{proof}

We note a corollary of the proof of \Cref{lem:compatiblesnakes}.

\begin{corollary}\label{cor:strongcolorchange}
  Pick a root polygon $\pgon$ and a compatible snake nest. Let 
  $\widehat H_\alpha,\widehat H_\beta$ be two (supporting hyperplanes 
  of) edges of $\polar$ that meet at a corner $c$ of $\polar$. Then the 
  snake nest crosses $\partial \polar $ at $c$ if and only if $\alpha$ 
  and $\beta$ have different colors. 
\end{corollary}

\begin{figure}
  \centering 
  \includegraphics[height=5.5cm]{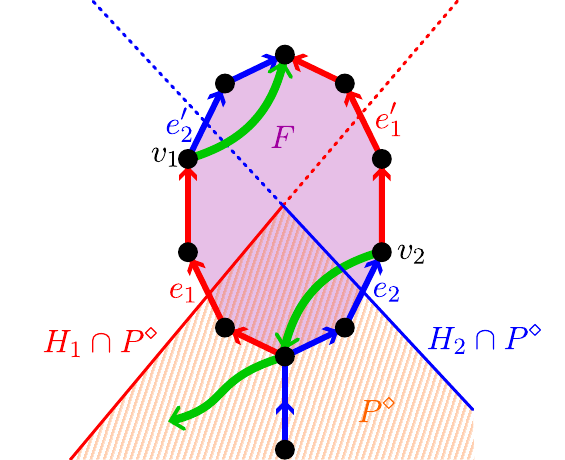}
  \caption{An illustration of how snake segments passing through 
  $\partial \polar $ must proceed from the exterior to the interior}
  \label{fig:proofcompatiblesnakes2}
\end{figure}

Given a snake nest compatible with a root polygon $\pgon$, we show 
in the following lemmas that a snake cannot travel from inside $\polar$ 
to outside $\polar$, and any snake that crosses $\partial \polar$ must 
terminate at $v_{\id}$. We then complete the proof of \cref{lem:polygon-lemma} 
by translating it into a statement about color changes on $\partial \polar$.

\begin{lemma}[Snakes Are Trapped]\label{how-snakes-can-cross-boundary}
  Choose a snake nest compatible with a root polygon $\pgon$. Then a 
  snake never travels from the interior of $\polar$ to the exterior of $\polar$. 
\end{lemma}

\begin{proof}
  See \cref{fig:proofcompatiblesnakes2}  for an illustration of this proof.  
  By definition of compatibility, a snake traveling from the interior of 
  $\polar$ to the exterior of $\polar$ would have to pass through a corner 
  of $\polar$ in some face $F$. A snake segment leaving $\polar$ must not 
  contain the source of $F$, since the snake segment containing the source 
  of $F$ is directed towards the source of $F$, which is inside $\polar$. 
  For there to be two snake segments in $F$, $F$ must be finite with color 
  changes at the non-source non-sink vertices $v_1,v_2$. Let $e_1,e_2$ be 
  the edges of $F$ that intersect $\partial \polar$.  The edges  $e_1$ and 
  $e_2$ must be distinct colors, since otherwise compatibility would imply 
  that neither snake segment in $F$ intersects $\partial \polar $. Let 
  $e_1',e_2'$ be the opposite edges of $F$. By \Cref{opposite-edges}, $e_1$ 
  and $e_1'$ (respectively, $e_2$ and $e_2'$) have the same color due to 
  both intersecting the same hyperplane $H_1$ (respectively $H_2$). It 
  follows that, without loss of generality, $v_1$ is between $e_1$ and 
  $e_2'$, and $v_2$ is between $e_2$ and $e_1'$. Convexity of $\polar$ then 
  implies that  $v_1$ and $v_2$ are not in $\polar$. This implies that the 
  source vertex is the only color change in its component of $F\setminus 
  \partial \polar$. Hence, the snake segment passing through the corner 
  must be directed toward the source. We conclude that the snake can never 
  travel from within $\polar$ to its exterior.
\end{proof}

\begin{lemma}[Return to the Source]\label{get-to-source-eventually}
  Choose a snake nest compatible with a root polygon $\pgon$. If a snake 
  crosses $\partial \polar$, traveling from outside $\polar$ to inside 
  $\polar$ within a face  $F$, then henceforth this snake stays inside 
  $\polar$ and terminates at the source vertex of $\Gamma$.
\end{lemma} 

\begin{proof}
  We proved in \cref{how-snakes-can-cross-boundary} that a snake can 
  never proceed across $\partial \polar$ from a point inside $\polar$ 
  to a point outside $\polar$, so all that remains is to show that a 
  snake that crosses $\partial\polar$ eventually reaches the source vertex 
  of $\Gamma$. By \cref{finite-regions}, there are finitely many chambers 
  contained in $\polar$, and hence by graph theoretic duality there are 
  also only finitely many faces in $\DD \setminus \Gamma $ that intersect 
  the interior of $\polar$.  Thus, after entering $\polar$ our snake may 
  only have finitely many segments within $\polar$.  It follows that the 
  snake must eventually terminate its travel due to only having finitely 
  many possible segments. By \cref{lem:orientation} and \cref{cor:endpoints}, 
  the only place the snake may terminate is at the source of $\Gamma$. 
\end{proof}

We are now ready to prove the Polygon Lemma.

\begin{proof}[Proof of \cref{lem:polygon-lemma}]
  Let $\pgon$ be a root polygon. We show that there are at most two color 
  changes in $\partial \pgon$. First we observe the following consequence 
  of the parabolic biclosed condition: if $\widehat \alpha$, $\widehat \beta$ 
  are vertices of $\pgon$ that are the endpoints of an edge $e$ of $\pgon$, 
  then there is a color change in $e$ if and only if $\widehat \alpha$ and 
  $\widehat\beta$ are different colors.  

  The dual of the line $L$ through $\widehat \alpha$ and $\widehat\beta$ is 
  the point $x$ where $L_\alpha$ and $L_\beta$ meet. When $L$ is disjoint 
  from $\DD$, the point $x$ is at the hyperbolic center of a finite face $F$ 
  inside $\DD$.  When $L$ is tangent to $\partial \DD$, the point $x$ is the 
  point of tangency and the ideal point of an infinite-face $F$ 
  (\Cref{lem:cusps}).
  Moreover, this sets up a bijection between color changes in $\partial \pgon$ 
  and color changes in the vertices of $\pgon$ (read in cyclic order). Hence, 
  to prove there are at most two color changes in $\partial \pgon$, it is 
  enough to show that there are at most two color changes among the vertices 
  of $\pgon$ read in cyclic order. Equivalently, it is enough to show that 
  the edges of the polar dual $\polar$ change colors at most two times when 
  read cyclically. By \Cref{cor:strongcolorchange}, this occurs if and only 
  if (fixing a compatible choice of snake nest for $\pgon$) the snake nest 
  crosses $\partial \polar$ at most twice. 

  By \Cref{get-to-source-eventually}, any snake that crosses 
  $\partial\polar$ eventually reaches the source of $\Gamma$. 
  However, by \cref{lem:orientation}, only two snake segments terminate at 
  the source.  This proves the snake nest crosses $\partial\polar$ at most 
  twice, as desired.
\end{proof}

\section{Coloring the Boundary}\label{sec:boundarycolor}
The purpose of this section is to pass from the Polygon Lemma 
(\cref{lem:polygon-lemma}), which concerns the vertices of a root polygon, 
to a statement about the boundary $\partial\mathbb D$. Specifically, this 
section identifies the separating line we will use later to prove biclosed 
sets are weakly separable. As before, we fix complementary parabolic-biclosed 
sets $B$ and $R$ with $B\sqcup R=\Phi^+$. We color a rescaled root 
$\widehat\beta$ blue (resp.\ red) if $\beta\in B$ (resp.\ $\beta\in R$). 
We now color a point $p\in\partial\mathbb D$ blue (resp.\ red) if it is an
accumulation point of blue (resp.\ red) rescaled roots. By \cref{cor:red-blue}
every point in $\partial\DD$ is either blue or red. We say a point in 
$\partial\DD$ is \dfn{purple} if it is both blue and red.  

Fix once and for all a clockwise orientation of the boundary circle 
$\partial\mathbb D$. Our goal is to show that the boundary cannot contain an 
alternating pattern of four points colored red, blue, red, blue in cyclic order.  

\begin{theorem}\label{boundary-coloring-main} 
  There do not exist four points $x_1,x_2,x_3,x_4\in\partial\mathbb D$ 
  that occur in this clockwise cyclic order and have colors 
  red, blue, red, blue, respectively. Moreover, there are at most two 
  purple points in $\partial\DD$. Each connected component of the
  complement of the purple points in $\partial \DD$ is monochromatic.     
\end{theorem}

Our strategy for proving \cref{boundary-coloring-main} is to assume 
$\partial\DD$ does contain an alternating pattern of four points colored 
red, blue, red, blue in cyclic order and then replace that pattern by a 
finite root polygon whose vertices include nearby rescaled roots of the 
same four colors.  The Polygon Lemma will then give a contradiction.  We 
next introduce a geometric device that lets us build such root polygons. 
For distinct points $p,q\in\partial\mathbb D$, write $[p,q]_{\mathrm{cw}}$ 
for the closed arc obtained by moving clockwise from $p$ to $q$, and write 
$(p,q)_{\mathrm{cw}}$ for its relative interior. 

\begin{definition}[Party Hat]\label{def:party-hats}
  Let $\beta\in\Phi^+$.  Since the rescaled root $\widehat\beta$ lies
  outside $\mathbb D$, there are exactly two tangent lines from
  $\widehat\beta$ to $\partial\mathbb D$.  Let $p_\beta$ and $q_\beta$ 
  denote the two points of tangency; we name these points so that $\widehat\beta$ 
  lies on the same side of $\widehat H_\beta$ as $(p_\beta,q_\beta)_{\mathrm{cw}}$.
  The \dfn{party hat} of $\beta$, denoted $\mathrm{PH}(\beta)$, is the
  closed region in $\mathbb A\setminus \mathbb D^\circ$ bounded by the two 
  tangent segments $[\widehat\beta,p_\beta]$ and $[\widehat\beta,q_\beta]$ and 
  by the closed boundary arc $[p_\beta,q_\beta]_{\mathrm{cw}}$. If we write
  $\Delta=\PH(\beta)$, then we denote the closed arc $[p_\beta,q_\beta]_{\mathrm{cw}}$ 
  and the open arc $(p_\beta,q_\beta)_{\mathrm{cw}}$ by $\Delta_{\arcc}$ and 
  $\Delta_{\arcc}^\circ$, respectively. A \dfn{party-hat cover} is a collection
  of positive roots whose party hats cover $\partial \DD $. 
\end{definition}

\cref{def:party-hats} is illustrated in \cref{fig:party_hat1}. 

\begin{figure}
  \centering
  \includegraphics[height=7.618cm]{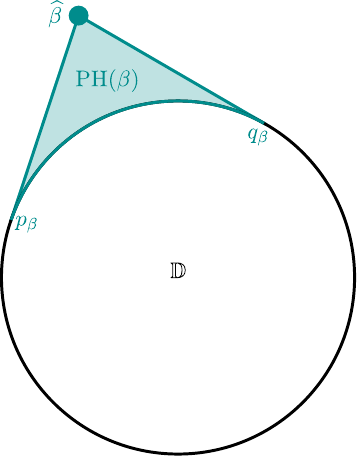}
  \caption{A party hat. } \label{fig:party_hat1}
\end{figure}

\begin{figure}
  \centering
  \includegraphics[height=8.785cm]{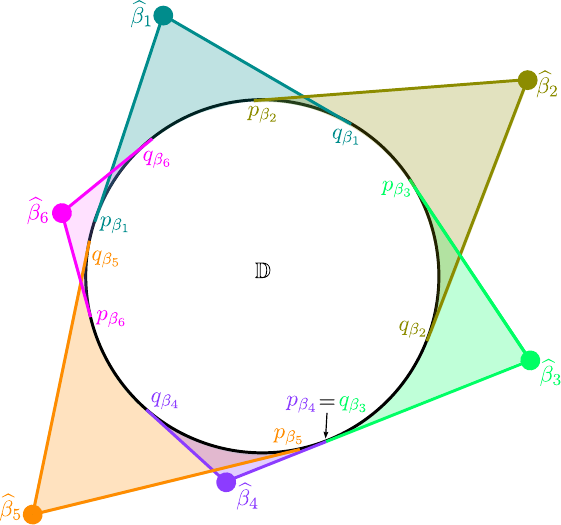}
  \caption{The party hats of the roots in a party-hat cover.  Since 
  $\widehat\beta_3$ and $\widehat\beta_4$ lie on a line tangent to 
  $\partial\DD$, the union $\PH(\beta_3,\beta_4)=\PH(\beta_3)\cup\PH(\beta_4)$ 
  is a double party hat. } \label{fig:party_hat_cover}
\end{figure}

\begin{remark}[Dual perspective]
  For a finite collection of positive roots, the corresponding open arcs 
  cover $\partial\DD$ if and only if the intersection of the closed half-spaces 
  containing the fundamental chamber is a polygon contained in $\DD^\circ$, 
  hence bounded in the hyperbolic metric. The closed arcs cover $\partial\DD$ 
  if and only if this intersection is a finite-area hyperbolic polygon, possibly 
  with ideal vertices.
\end{remark}

Given a point $p\in\partial\DD$, we would like to replace $p$ with a rescaled 
root $\widehat\beta$ such that $p$ lies in the open arc $(p_\beta,q_\beta)_{\cw}$; 
the openness is crucial here because we will want to apply a compactness argument. 
There is one potential problem: such a root $\beta$ might not exist. Indeed, this 
can be the case if $p$ is the accumulation point of the rescaled positive roots 
in a rank 2 parabolic subsystem of $\Phi$. The following definition will allow us 
to overcome this hurdle. 

\begin{definition}[Double Party Hat] \label{def:doublepartyhat}
  Suppose $\beta_1$ and $\beta_2$ are two roots such that the line through 
  $\widehat \beta_1$ and $\widehat\beta_2$ is tangent to $\partial\DD$. Let
  $\PH(\beta_1,\beta_2)=\PH(\beta_1)\cup\PH(\beta_2)$; we call $\PH(\beta_1,\beta_2)$ 
  a \dfn{double party hat}. If we write $\Delta=\PH(\beta_1,\beta_2)$, then 
  we denote the closed arc $[p_{\beta_1},q_{\beta_1}]_{\cw} 
  \cup[p_{\beta_2},q_{\beta_2}]_{\cw}$ by $\Delta_{\arcc}$, and we let 
  $\Delta_{\arcc}^\circ$ denote the relative interior of $\Delta_{\arcc}$. 
\end{definition}

The following simple observation about party hats is fundamental to several of 
our arguments. 

\begin{observation}\label{hat-containment}
  For $y\in\partial\DD$ and $\beta\in\Phi^+$, we have $y\in\PH(\beta)$ if 
  and only if $\widehat\beta$ is in the closed half-space bounded by the tangent
  line $L_y$ that is disjoint from $\DD^\circ$.  
\end{observation}

\begin{lemma}\label{lem:party-hat-cover-convex-hull}
  Let $\Lambda\subseteq\Phi^+$ be finite, and set
  $\pgon=\operatorname{Conv}\widehat\Lambda$. Then $\Lambda$ is a party-hat 
  cover if and only if $\mathbb D\subseteq \pgon$. Suppose in addition that 
  $\Omega\subseteq\Lambda$ and that for every $\beta\in\Omega$, there is 
  a point $x_\beta\in[p_\beta,q_\beta]_{\cw}$ that belongs to no other
  arc $[p_\gamma,q_\gamma]_{\cw}$ with $\gamma\in\Lambda\setminus\{\beta\}$. 
  Then $\widehat\Omega$ is contained in the set of vertices of $\pgon$. The 
  cyclic order of these vertices agrees with the cyclic order of the points 
  $x_\beta$. 
\end{lemma}

\begin{proof}
  The polygon $\pgon$ contains $\mathbb D$ if and only if for every
  tangent line $L$ to $\partial\mathbb D$, the set $\pgon$ intersects the closed
  half-plane bounded by $L$ that is disjoint from $\mathbb D^\circ$.
  This occurs if and only if each point in $\partial\DD$ belongs to one of the arcs
  $[p_\beta,q_\beta]_{\cw}$ with $\beta\in \Lambda$.

  Now choose $\beta\in\Omega$, and let $T_{x_\beta}$ be the line tangent to 
  $\partial\DD$ at $x_{\beta}$. By hypothesis, $x_\beta$ does not belong to any 
  of the party hats of the roots in $\Lambda\setminus\{\beta\}$. This implies 
  that the rescaled roots in $\widehat\Lambda\setminus\{\widehat\beta\}$ all lie 
  on the opposite side of $T_{x_\beta}$ from $\widehat\beta$. Hence, $\widehat\beta$ 
  is a vertex of $\pgon$. The assertion about the cyclic order is straightforward.
\end{proof}

\begin{lemma}\label{lem:Fy}
  Suppose $y\in\partial\DD$ and $U$ is an open subset of $\mathbb A$ containing $y$. 
  Then there exists a party hat or double party hat $\Delta$ contained in $U$ such 
  that $y$ belongs to the open arc $\Delta_{\arcc}^\circ$.  
\end{lemma}

\begin{proof}
  Let $L=L_y$ be the line tangent to $\partial\DD$ at $y$.
  If $\resroots'=\resroots\cap L$ has at least two roots, 
  then $\Phi'$ is a parabolic rank $2$ subsystem by \Cref{lem:cusps}.
  In this case, we may choose rescaled roots 
  $\widehat{\gamma}_1,\widehat{\gamma}_2$ in $L$ on either 
  side of $y$ such that 
  $[p_{\gamma_1},q_{\gamma_1}]_{\cw}\cup 
  [p_{\gamma_2},q_{\gamma_2}]_{\cw}\subseteq U$ and 
  $q_{\gamma_1}=y=p_{\gamma_2}$.
  Moreover, we can assume $\PH(\gamma_1)$ and $\PH(\gamma_2)$ 
  are contained in $U$. Then $y$ is contained in the open arc 
  $\Delta_{\arcc}^{\circ}$ of the double party hat $\Delta = 
  \PH(\gamma_1,\gamma_2)$.

  Now assume $L$ has at most one rescaled root. By \cref{prop:finite_visibility}, 
  the open half plane $L^+$ on the side of $L$ not containing $\DD$ includes an 
  infinite set of rescaled positive roots. These roots have an accumulation 
  point at $y$, so we may choose a rescaled root $\widehat{\gamma}\in L^+$ 
  sufficiently close to $y$ to ensure $\PH(\gamma)\subseteq U$. In this case, 
  we take $\Delta=\PH(\gamma)$.
\end{proof}

\begin{lemma}\label{lem:extend-prescribed-party-hats}
  Let $\Omega\subseteq\Phi^+$ be a finite set of positive roots, and assume 
  that the arcs $[p_\beta,q_\beta]_{\cw}$ for $\beta\in\Omega$ are pairwise 
  disjoint. Then there is a finite party-hat cover $\Lambda\subseteq\Phi^+$ 
  containing $\Omega$ such that for every $\beta\in\Omega$, there is a point
  $x_\beta\in(p_\beta,q_\beta)_{\cw}$ that belongs to no arc 
  $[p_\gamma,q_\gamma]_{\cw}$ with $\gamma\in\Lambda\setminus\{\beta\}$. 
\end{lemma}

\begin{proof}
  For each $\beta\in \Omega$, choose a point $x_\beta\in (p_\beta,q_\beta)_{\cw}$. 
  Note that $x_\beta\not\in[p_{\beta'},q_{\beta'}]_{\cw}$ for all distinct 
  $\beta,\beta'\in \Omega$. Let $X=\{x_\beta:\beta\in\Omega\}$.

  For each $y\in\partial\mathbb D\setminus X$, choose an open boundary arc
  $U_y$ that contains $y$ and is disjoint from $X$. Choose an open neighborhood 
  $V_y$ of $y$ in $\AA$ such that $V_y\cap\partial\DD\subseteq U_y$. Applying 
  \cref{lem:Fy} to $V_y$ gives a party hat or double party hat $\Delta_y\subseteq
  V_y$ such that 
  \[
  y\in(\Delta_y)_{\arcc}^\circ\subseteq(\Delta_y)_{\arcc}\subseteq U_y.
  \]
  By construction, $(\Delta_y)_{\arcc}$ is disjoint from $X$. The collection 
  \[
  \mathcal C=\{(p_\beta,q_\beta)_{\cw}:\beta\in\Omega\}\cup 
  \{(\Delta_y)_{\arcc}^\circ:y\in\partial\DD\setminus X\}
  \] 
  is an open cover of $\partial\DD$. Every subcover of $\mathcal C$ must contain 
  $\{(p_\beta,q_\beta)_{\cw}:\beta\in\Omega\}$ because for each $\beta\in\Omega$, 
  the only set in $\mathcal C$ containing $x_\beta$ is $(p_\beta,q_\beta)_{\cw}$. 
  Since $\partial\DD$ is compact, $\mathcal C$ has a finite subcover. This finite 
  subcover has the form \[\{(p_\beta,q_\beta)_{\cw}:\beta\in\Omega\}\cup 
  \{(\Delta_{y_1})_{\arcc}^\circ,\ldots,(\Delta_{y_m})_{\arcc}^\circ\}\] for 
  some $y_1,\ldots,y_m\in\partial\DD\setminus X$. For $1\leq i\leq m$, note that 
  $\Delta_{y_i}$ is either a party hat $\PH(\beta)$ or a double party hat 
  $\PH(\beta_1,\beta_2)$; let 
  \[
  F_i=
  \begin{cases}
    \{\beta\} & \text{if } \Delta_{y_i}=\PH(\beta); \\
    \{\beta_1,\beta_2\} & \text{if } \Delta_{y_i}=\PH(\beta_1,\beta_2).
  \end{cases}
  \]
  The set
  \[
  \Lambda=\Omega \cup \bigcup_{i=1}^m F_{i}
  \]
  is our desired finite party-hat cover. 
\end{proof}

\begin{lemma}\label{lem:prescribed-party-hats-give-polygon}
  Let $\alpha_1,\ldots,\alpha_k\in\Phi^+$. Suppose that the closed arcs
  $[p_{\alpha_1},q_{\alpha_1}]_{\cw},\ldots,
  [p_{\alpha_k},q_{\alpha_k}]_{\cw}$
  are pairwise disjoint and occur in this clockwise cyclic order on
  $\partial\mathbb D$. Then there is a root polygon $\pgon$ such that
  $\widehat{\alpha}_1,\ldots,\widehat{\alpha}_k$ are vertices of $\pgon$ and
  occur in this clockwise cyclic order along $\partial \pgon$.
\end{lemma}

\begin{proof}
  By \cref{lem:extend-prescribed-party-hats}, there is a finite
  party-hat cover $\Lambda$ containing $\alpha_1,\ldots,\alpha_k$ 
  such that for each $i\in[k]$, there is a point 
  $x_i\in(p_{\alpha_i},q_{\alpha_i})_{\cw}$ that does not belong to 
  any of the arcs $[p_{\gamma},q_{\gamma}]_{\cw}$ for 
  $\gamma\in\Lambda\setminus\{\alpha_i\}$. Set 
  $\pgon=\conv\{\widehat\beta:\beta\in\Lambda\}$. By 
  \cref{lem:party-hat-cover-convex-hull}, the polygon $\pgon$ 
  contains $\mathbb D$, and each rescaled root $\widehat{\alpha}_i$ 
  is a vertex of $\pgon$. Furthermore, these vertices occur in the same 
  cyclic order as the points $x_1,\ldots,x_k$.
\end{proof}

We can now prove the main result of this section. 

\begin{proof}[Proof of \cref{boundary-coloring-main}]
  We claim that there do not exist distinct points $x_1,x_2,x_3,x_4$ 
  appearing in this clockwise cyclic order along $\partial\DD$ such that 
  $x_1,x_3$ are red and $x_2,x_4$ are blue. Suppose by way of contradiction 
  that such points $x_1,x_2,x_3,x_4$ exist. Choose $\varepsilon>0$ sufficiently 
  small, and let $U_i$ be the open disc in $\mathbb A$ of radius $\varepsilon$ 
  with center $x_i$. We may assume $\varepsilon$ is small enough that
  $U_1,U_2,U_3,U_4$ are pairwise disjoint. As a rescaled root approaches a 
  point of $\partial\DD$, both of its tangent points approach that point.
  Since each $x_i$ is an accumulation point of roots of its prescribed color, 
  we may therefore choose red rescaled roots $\widehat\alpha_1\in U_1$ and 
  $\widehat\alpha_3\in U_3$ and blue rescaled roots $\widehat\alpha_2\in U_2$ 
  and $\widehat\alpha_4\in U_4$ such that $[p_{\alpha_i},q_{\alpha_i}]_{\cw}
  \subseteq U_i$ for each $i\in\{1,2,3,4\}$. This implies that the closed 
  party-hat arcs of $\alpha_1,\alpha_2,\alpha_3,\alpha_4\in\Phi^+$ are pairwise 
  disjoint. \cref{lem:prescribed-party-hats-give-polygon} now produces a root
  polygon $\pgon$ in which $\widehat{\alpha}_1,\widehat{\alpha}_2,
  \widehat{\alpha}_3,\widehat{\alpha}_4$ occur as vertices in this cyclic order. 
  Each of the four portions of $\partial \pgon$ between consecutive prescribed 
  vertices contains a color change, contradicting the Polygon Lemma 
  (\cref{lem:polygon-lemma}). This proves the first statement of the theorem. 

  The preceding argument readily implies that there are at most two purple points 
  in $\partial\DD$. Indeed, if there are three purple points $z_1,z_2,z_3$ with 
  $z_2\in(z_1,z_3)_{\cw}$, then we can choose a fourth point $z_4\in(z_3,z_1)_{\cw}$. 
  If $z_4$ is red, then viewing $z_1,z_2,z_3$ as blue, red, blue yields a contradiction; 
  if $z_4$ is blue, then viewing $z_1,z_2,z_3$ as red, blue, red yields a contradiction. 
  It is similarly straightforward to check that each connected component of the 
  complement of the purple points in $\partial\DD$ must be monochromatic. 
\end{proof}

\section{Biclosed implies weakly separable} \label{sec:FinishProof}
In this section we will make use of the induced coloring 
of $\partial \DD$ to construct separating lines for biclosed 
sets. In particular, here is where we first begin using the 
full power of the biclosed condition, rather than just the 
weaker parabolic biclosed condition used for studying the coloring.

\begin{figure}
  \centering
  \includegraphics[height=10.574cm]{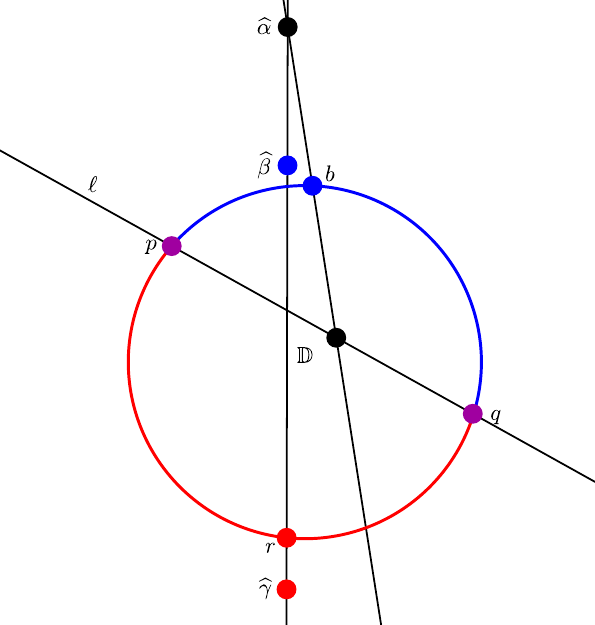}
  \caption{Illustrating the proof of \Cref{lem:two-purple}.}
  \label{fig:twopurple}
\end{figure}

The following lemma is proven similarly to \cite[Lemma 5.4]{BDHMMS25}, 
which covers the special case where $W=U_3$. 

\begin{lemma}\label{lem:two-purple}
  If $B$ is a biclosed set so that $\partial \DD $ has exactly two purple 
  points, then $B$ is weakly separable with the line through these two 
  points serving as the separating hyperplane.
\end{lemma}

\begin{proof}
  Let $\Phi^+ = B\sqcup R$ be the decomposition of the positive roots 
  into two biclosed sets so that the induced coloring of $\partial\DD$ 
  has exactly two purple points, $p$ and $q$.  Let $\ell$ be the line 
  through those two points. By \cref{boundary-coloring-main}, the two 
  components of $\partial\DD\setminus\{p,q\}$ are monochromatic and have 
  different colors: otherwise, a point in each component, together with 
  the two purple points, would give an alternating four-point pattern. 
  Hence the two open half-spaces bounded by $\ell$ can be labeled either 
  blue or red according to which component of $\partial\DD\setminus\{p,q\}$ 
  it contains. We will show that if $\widehat\alpha$ is a rescaled root in 
  the blue half-space, then $\widehat\alpha$ is blue (i.e., in $B$). The 
  proof for the red half-space is similar; hence this will show that $\widehat B$ 
  is weakly separated by $\ell$. See \Cref{fig:twopurple}. 
  
  Let $\widehat\alpha$ be in the blue half-space. Let $m$ be the midpoint 
  of the line segment $\overline{pq}$. Then the line 
  $\overleftrightarrow{\widehat\alpha m}$ intersects $\partial\DD$ in a blue 
  point and a red point. Call the blue point $b$. Since it is blue, there 
  exist blue rescaled roots arbitrarily close to $b$. Hence we may find a 
  blue rescaled root $\widehat\beta$ so that the line connecting 
  $\widehat\alpha$ and $\widehat\beta$ is very close 
  to $\overleftrightarrow{\widehat\alpha m}$; in particular, we choose 
  $\widehat\beta$ so that the line connecting $\widehat \alpha$ and $\widehat\beta$ 
  intersects $\partial \DD$ at a red point $r$. The roots on the line connecting 
  $\widehat\alpha$ and $\widehat\beta$ form a rank 2 root system $\resroots'$. 
  By \Cref{rem:rank-2}, the point $r$ is an accumulation point for $\resroots'$. 
  Since $r$ is red and not purple, there must be (infinitely many) rescaled roots 
  in $\resroots'$ that are red. Choose one of these red rescaled roots $\widehat\gamma$ 
  sufficiently close to $r$ that $\widehat\beta$ lies between $\widehat\alpha$ and 
  $\widehat\gamma$. Then $\beta$ is a non-negative combination of $\alpha$ and $\gamma$, 
  so the fact that $B$ is biclosed, $\beta$ is blue, and $\gamma$ is red implies that 
  $\alpha$ is blue. This proves the claim. The restriction of $B$ to the two-dimensional 
  linear hyperplane corresponding to $\ell$ is separable by the rank-two characterization 
  of biclosedness. Hence, $B$ is weakly separable. 
\end{proof}

\begin{lemma}\label{lem:chainbelowinfinite}
  If $B$ is an infinite biclosed set, then there exists
  \[ B_0 \subset B_1 \subset \cdots \subset B \]
  so that each $B_i$ is biclosed with $|B_i|=i$.
\end{lemma}

\begin{proof}
  Each nonempty biclosed set contains a simple root $\alpha_i$. 
  Inductively construct a sequence of simple roots so that 
  $\alpha_{i_1}\in B$, and $\alpha_{i_{k+1}} \in s_{i_k}\cdots 
  s_{i_1}\cdot B$. Then define $B_k = \Inv(s_{i_1}\cdots s_{i_k})$. 
  We have $|B_k|=k$ and $B_k\subset B_{k+1} \subset B$, so the claim 
  follows.
\end{proof}

We now show that the biclosed sets that are unions of an increasing 
sequence of finite inversion sets are separated by a tangent line to $\DD$.

\begin{lemma}
  Let $B = \bigcup_{i\in \NN} B_i$, where 
  \[ B_0 \subset \cdots \subset B_k \subset \cdots \]
  is a filtration by biclosed sets with $|B_i|=i$. Then $B$ is weakly 
  separated by a tangent line to $\DD$.
\end{lemma}

\begin{proof}
  By \cite[Proposition 4.2]{HohlwegLabbe}, $B$ is weakly separable. 
  Let $H$ be a separating hyperplane. It must be the case that $\widehat H$ 
  intersects $\DD$, for otherwise $B$ would be finite (clearly impossible) or 
  cofinite. The latter case is also impossible, since any cofinite set contains
  two roots $\alpha,\beta$ such that the line segment connecting 
  $\widehat\alpha,\widehat\beta$ intersects $\DD$, but no finite biclosed set 
  has that property.
    
  Color the roots in $B$ blue and the remaining roots red. If $\widehat H$ 
  intersects $\DD$ in points $p_1\neq p_2$, then there exist blue rescaled 
  roots $\widehat\alpha,\widehat\beta$ with $\widehat\alpha$ very close to 
  $p_1$ and $\widehat\beta$ very close to $p_2$. In particular, we may pick 
  $\widehat\alpha$ and $\widehat\beta$ so that the line spanned by $\widehat\alpha$ 
  and $\widehat\beta$ intersects $\DD$. But there is some finite $B_i\subset B$ 
  so that $\alpha,\beta\in B_i$. Since $B_i$ is a finite biclosed set, we have 
  $\mathrm{conv} \widehat B_i \cap \DD =\varnothing$. This is a contradiction, 
  so we must have $p_1=p_2$.
\end{proof}

\begin{figure}
  \centering
  \includegraphics[height=8.508cm]{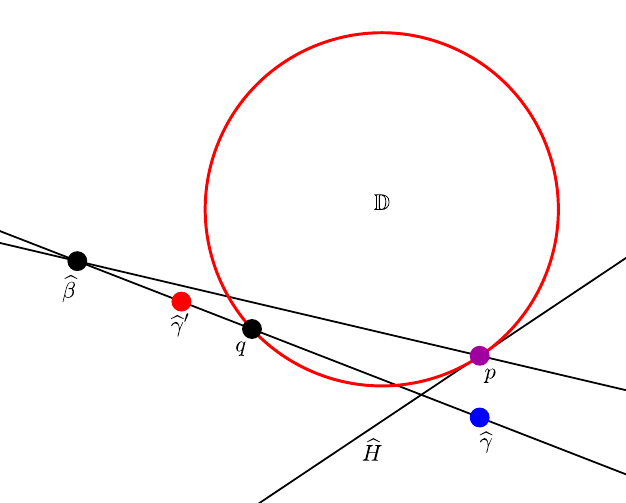}
  \caption{Illustrating the proof of \Cref{lem:sep-plane-if-one-purple}. }
  \label{fig:onepurple}
\end{figure}

\begin{lemma}\label{lem:sep-plane-if-one-purple}
  If $\Phi^+ = B\sqcup R$ are biclosed sets so that 
  $\partial\DD $ has all points red except for a single point $p$  that is purple, then 
  $B$ is weakly separable with the tangent line at $p$ as the separating hyperplane.
\end{lemma}

\begin{proof}
  Let $B_0\subset B_1 \subset \cdots \subset B$ be the filtration from
  \Cref{lem:chainbelowinfinite}, and set $N= \bigcup_{k\in \NN} B_k$. 
  Let $H$ be the weakly separating hyperplane of $N$ so that $\widehat H$ 
  is tangent to $\partial \DD$. Every neighborhood of the point of tangency 
  for $\widehat H$ has infinitely many roots from $N$. In particular, that 
  point in $\partial\DD$ is blue or purple, so that point must in fact be $p$. 
  See \Cref{fig:onepurple}. 
    
  We have $N\subseteq B$. So let $\beta \in \Phi^+\setminus (N \cup H)$. 
  We will show $\beta \in R$, implying that $H$ is a weakly separating 
  hyperplane for $B$. Since $\beta$ is not on $H$, the line connecting 
  $\widehat\beta$ with $p$ must intersect $\partial\DD$ twice. There are
  infinitely many blue rescaled roots in a neighborhood of $p$; let 
  $\widehat\gamma$ be blue and close enough to $p$ so that the line 
  segment connecting $\widehat \beta$ to $\widehat \gamma$ intersects 
  $\partial\DD$ twice. Let $q$ be an intersection point distinct from $p$. 
  Then in the rank two subsystem $\Phi'$ spanned by $\beta$ and $\gamma$, 
  there are infinitely many rescaled roots in any neighborhood of $q$. 
  Since $q$ is colored red, it follows from \Cref{rem:rank-2}
  that there is some root $\gamma'$ in $\Phi'$ that is red; by choosing 
  a rescaled root close enough to $q$, we can guarantee that $\gamma'$ 
  is in the positive span of $\beta$ and $\gamma$. Hence, by the biclosed 
  condition, $\beta$ is red. The restriction of $B$ to $H$ is separable by 
  the rank-two characterization of biclosedness, completing the proof of 
  weak separability.
\end{proof}

\begin{remark}
  In the universal group, our proof for this case worked for parabolic 
  biclosed sets, but here we are using the full biclosed condition. We 
  do not know if there can be a parabolic biclosed set with a unique 
  purple point that is not weakly separable.
\end{remark} 

\begin{remark}
  \Cref{lem:sep-plane-if-one-purple} also handles the case where all 
  points in $\partial\DD$ are blue except one purple point.
\end{remark}

\begin{lemma}\label{lem:nopurple}
  If $B$ is a biclosed set so that $\partial\DD$ has no purple points, 
  then $B$ is either finite or cofinite.
\end{lemma}

\begin{proof}
  Without loss of generality, every point of $\partial \DD $ is red. 
  If $B$ were infinite, then it would have an accumulation point on 
  $\partial\DD$, which must then be purple. This is a contradiction, 
  so $B$ is finite.
\end{proof}

\begin{theorem}[Weakly Separable]\label{thm:separable}
  Let $W$ be a rank~$3$ Coxeter group with root system $\Phi$.  
  Every biclosed subset of $\Phi^+$ is weakly separable.
\end{theorem}

\begin{proof}
  By \cite[Remark 3.10]{BDHMMS25}, the claim depends only on $W$ and 
  not the choice of realization for the root system $\Phi$. If $W$ is 
  reducible, its root system is a direct sum of systems of ranks at 
  most $2$; a biclosed set restricts to a separable set in each summand, 
  and the separating functionals extend by direct sum. Thus, we may assume 
  that $W$ is irreducible. If $W$ is finite or affine, then the result is 
  shown in \cite[Theorem B]{barkley:2023affine}. Otherwise, by 
  \Cref{rem:mostlyhyperbolic}, $W$ is hyperbolic. In that case, taking a 
  hyperbolic realization for $\Phi^+$ as in \Cref{def:roots}, we are in 
  the setup of this paper. 

  Let $\Phi^+ = B\sqcup R$ be a decomposition of the positive roots into 
  biclosed sets. Then there is an induced coloring of the points of 
  $\partial \DD$ as described in \Cref{sec:boundarycolor}. If there are no 
  purple points on $\partial\DD$, then $B$ is finite or cofinite by 
  \Cref{lem:nopurple}. In that case, \Cref{cor:finitecofinite} implies 
  that $B$ is weakly separable. Otherwise, $\partial\DD$ has at least one 
  purple point. By \Cref{boundary-coloring-main}, there is either exactly 
  one purple point or exactly two purple points. In those cases, 
  \Cref{lem:sep-plane-if-one-purple} and \Cref{lem:two-purple} respectively 
  show that $B$ is weakly separable.
\end{proof}

\section{Consequences of the Main Result} \label{sec:Consequences}

Here we briefly give consequences of our result for all rank 3 
Coxeter groups that biclosed sets are weakly separable.  First, 
we show that Dyer's lattice conjecture holds in rank 3.

\begin{corollary} \label{cor:lattice}
  For any Coxeter group of rank 3, its extended weak order is a 
  complete lattice.
\end{corollary}

\begin{proof}
  This follows from \Cref{thm:separable} and 
  \cite[Corollary 2.36]{LabbeThesis}, which asserts that the weakly 
  separable sets in a rank 3 root system form a lattice. A directed 
  union of biclosed sets is biclosed, so the union of the joins of 
  finite subfamilies gives the join of any family. Taking complements 
  gives arbitrary meets. 
\end{proof}

To use the separability of biclosed sets to prove Dyer's Conjecture A, 
we will employ the following lemma about weakly separable sets. This 
was essentially proved in \cite[Proposition 7.3]{BDHMMS25}, but here we 
give a slightly different (and more general) proof.

\begin{lemma}\label{lem:separablecover}
  Let $\Phi^+$ be the set of positive roots of a reduced Coxeter root 
  system spanning an $n$-dimensional vector space $V$, and let 
  $B_1\lessdot B_2$ be a cover relation of weakly separable sets. 
  Then $|B_2\setminus B_1| =1$.
\end{lemma}

\begin{proof}
  By the hyperplane separation theorem, a biclosed set $B$ is weakly 
  separable if and only if there are linear functionals $f_1,\ldots,f_n 
  : V\to \RR$ such that $B$ is exactly the set of roots $\alpha\in \Phi^+$ 
  so that 
  \[ (f_1(\alpha),f_2(\alpha),\ldots,f_n(\alpha)) < (0,0,\ldots, 0) \]
  in the lexicographic order on tuples of real numbers. Let $B <B'$ be 
  weakly separable sets, and let $(f_i)$ and $(f_i')$ be their separating 
  tuples. For $\lambda \in [0,1]$, define the functional 
  $f^\lambda_i=(1-\lambda)f_i + \lambda f_i'$. We can form the tuple of 
  functions $(f^\lambda_i)$. Each such tuple forms a weakly separable set 
  $B_\lambda$ by taking the set of roots whose values are less than 
  $(0,0,\ldots,0)$. If $\lambda\leq \lambda'$, then $B_{\lambda}\leq B_{\lambda'}$. 
  In particular, we have $B\leq B_\lambda \leq B'$ for all $\lambda$.
    
  Assume $\alpha,\beta \in B'\setminus B$ are two distinct roots. We will 
  show there is a weakly separable set $B''$ so that $B<B'' < B'$. This will 
  prove the lemma, by contraposition. If there exists a $\lambda \in [0,1]$ 
  so that $|\{\alpha,\beta\}\cap B_\lambda|=1$, then $B_\lambda$ is distinct 
  from $B$ and $B'$, so it satisfies the claim. Otherwise, let $\lambda$ be 
  the supremum of $\{\lambda\mid \{\alpha,\beta\}\cap B_\lambda=\varnothing\}$. 
  Then $f^\lambda_i(\alpha)=f^{\lambda}_i(\beta)=0$ for all $i$.  Let $\Phi'$ 
  be the root subsystem consisting of all roots $\gamma$ so that 
  $f^\lambda_i(\gamma)=0$ for all $i$. Let $V'$ be the span of $\Phi'$. Take 
  any functional $g:V'\to \RR$ such that $g(\alpha)<0$ and $g(\beta)>0$. Then 
  define the set of positive roots
  \[ B'' \coloneqq B_\lambda\cup \{\gamma\in \Phi'^+\mid g(\gamma)<0\}. \]
  Then $B''$ is weakly separable and satisfies $B < B'' < B'$, as desired.
\end{proof}

Now we can deduce Dyer's Conjecture A for rank 3 Coxeter groups.

\begin{corollary}
  For any Coxeter group of rank 3, if $B_1\lessdot B_2$ is a 
  cover relation in extended weak order, then $|B_2\setminus B_1| = 1$. 
\end{corollary}

\begin{proof}
  By \Cref{thm:separable}, a cover relation $B_1\lessdot B_2$ 
  of biclosed sets is the same as a cover relation of weakly 
  separable sets. Hence by \Cref{lem:separablecover}, 
  $|B_2\setminus B_1|=1$.
\end{proof}

\section*{AI Statement}
We did not use AI in any manner in producing this paper. 

\section*{Acknowledgments}
We thank the Banff International Research Station (BIRS) for 
hosting a workshop where this research project began, in particular 
bringing all of us to a wonderful work environment for a week.   
We especially thank the organizers of the BIRS workshop on Lattice 
Theory, namely Nathan Williams, Emily Barnard, Cesar Ceballos, 
Colin Defant, and Osamu Iyama, as well as conference participants  
Nathan Reading and Oliver Pechenik for very helpful conversations. 
We also thank the Institute for Computational and Experimental 
Research in Mathematics (ICERM) for running a workshop during which 
some of the later work on this project was done.  

GTB was supported by NSF grant DMS-2503536. 
CD was supported by a Benjamin Peirce Fellowship at Harvard University.  
PH was supported by ICERM grant DMS-2424556 and  Simons Foundation 
travel funding for mathematicians.  
JM was partially supported by NSF grant DMS-2204001. 
DES was partially supported by NSF grant DMS-2246570.
The opinions, findings, conclusions and recommendations expressed in 
this material are those of the authors and do not necessarily reflect 
the views of the National Science Foundation.

\bibliographystyle{alpha}
\bibliography{main}
\end{document}

%% file: rootFig.tex
\begin{tikzpicture}[scale=.5]
\draw[fill, blue] (-8.66025, -5) circle (.05);
\draw[fill, red] (0, 10) circle (.05);
\draw[fill, blue] (8.66025, -5) circle (.05);
\draw[fill, blue] (3.06016e-16, -5) circle (.05);
\draw[fill, red] (5.56944, 0.353438) circle (.05);
\draw[fill, red] (3.09081, 4.64656) circle (.05);
\draw[fill, red] (1.60979e-16, 2.10929) circle (.05);
\draw[fill, blue] (-4.70053e-16, -1.74188) circle (.05);
\draw[fill, red] (3.85417, 3.32437) circle (.05);
\draw[fill, red] (4.80608, 1.67563) circle (.05);
\draw[fill, blue] (3.21701, -0.531586) circle (.05);
\draw[fill, blue] (-3.69872e-16, 0.760644) circle (.05);
\draw[fill, red] (4.33013, 2.5) circle (.05);
\draw[fill, blue] (-5.22306e-16, -0.706876) circle (.05);
\draw[fill, blue] (1.187, -1.29531) circle (.05);
\draw[fill, red] (2.34567, 2.32094) circle (.05);
\draw[fill, blue] (2.40304, -0.837814) circle (.05);
\draw[fill, blue] (-5.57627e-16, -4.14963e-16) circle (.05);
\draw[fill, red] (0.907509, 1.3643) circle (.05);
\draw[fill, red] (3.47249, 1.01453) circle (.05);
\draw[fill, blue] (1.8267, -1.05464) circle (.05);
\draw[fill, red] (1.7853, 1.9482) circle (.05);
\draw[fill, blue] (2.78472, -0.176719) circle (.05);
\draw[fill, blue] (0.658737, -0.380322) circle (.05);
\draw[fill, red] (3.15959, 1.82419) circle (.05);
\draw[fill, red] (1.37554, 1.67563) circle (.05);
\draw[fill, red] (3.1912, 0.527322) circle (.05);
\draw[fill, blue] (1.35194, -0.780541) circle (.05);
\draw[fill, red] (2.21385, 1.90953) circle (.05);
\draw[fill, blue] (0.612172, 0.353438) circle (.05);
\draw[fill, blue] (2.99654, 0.190161) circle (.05);
\draw[fill, blue] (0.920734, -0.531586) circle (.05);
\draw[fill, red] (2.7416, 1.86191) circle (.05);
\draw[fill, red] (1.1031, 1.20374) circle (.05);
\draw[fill, red] (3.18282, 0.870938) circle (.05);
\draw[fill, blue] (2.03916, -0.595766) circle (.05);
\draw[fill, blue] (1.11876, -0.645916) circle (.05);
\draw[fill, red] (2.47864, 1.88564) circle (.05);
\draw[fill, blue] (0.812956, 0.701205) circle (.05);
\draw[fill, red] (3.1708, 1.3643) circle (.05);
\draw[fill, blue] (2.66717, -0.0802208) circle (.05);
\draw[fill, blue] (0.837159, -0.291874) circle (.05);
\draw[fill, red] (2.90227, 1.67563) circle (.05);
\draw[fill, red] (1.61665, 1.59961) circle (.05);
\draw[fill, red] (0.951902, 0.941868) circle (.05);
\draw[fill, red] (3.17706, 1.10767) circle (.05);
\draw[fill, blue] (2.3016, -0.380322) circle (.05);
\draw[fill, blue] (0.719479, 0.0456583) circle (.05);
\draw[fill, blue] (1.39998, -0.630593) circle (.05);
\draw[fill, red] (2.17378, 1.78448) circle (.05);
\draw[fill, blue] (0.77693, 0.448561) circle (.05);
\draw[fill, red] (3.02495, 1.53339) circle (.05);
\draw[fill, blue] (2.94666, 0.353438) circle (.05);
\draw[fill, blue] (2.47864, -0.234992) circle (.05);
\draw[fill, blue] (0.780355, -0.128948) circle (.05);
\draw[fill, blue] (1.75344, -0.611334) circle (.05);
\draw[fill, red] (1.84574, 1.67563) circle (.05);
\draw[fill, red] (1.16511, 1.15283) circle (.05);
\draw[fill, red] (3.09081, 0.825322) circle (.05);
\draw[fill, blue] (2.09946, -0.465509) circle (.05);
\draw[fill, blue] (0.750268, 0.261579) circle (.05);
\draw[fill, blue] (1.1031, -0.496868) circle (.05);
\draw[fill, red] (2.51383, 1.77042) circle (.05);
\draw[fill, blue] (1.57632, -0.620985) circle (.05);
\draw[fill, red] (2.00328, 1.7279) circle (.05);
\draw[fill, red] (1.42022, 1.40525) circle (.05);
\draw[fill, red] (3.00396, 0.541012) circle (.05);
\draw[fill, blue] (2.62759, -0.0477244) circle (.05);
\draw[fill, blue] (0.897176, -0.262121) circle (.05);
\draw[fill, red] (2.8203, 1.6283) circle (.05);
\draw[fill, blue] (1.94267, -0.531586) circle (.05);
\draw[fill, red] (1.68981, 1.57655) circle (.05);
\draw[fill, blue] (0.963543, 0.831093) circle (.05);
\draw[fill, red] (3.09081, 1.13486) circle (.05);
\draw[fill, red] (1.29412, 1.28047) circle (.05);
\draw[fill, red] (3.04493, 0.675145) circle (.05);
\draw[fill, blue] (2.80724, 0.178148) circle (.05);
\draw[fill, blue] (1.01958, -0.401658) circle (.05);
\draw[fill, red] (2.63075, 1.7162) circle (.05);
\draw[fill, blue] (1.41649, -0.579064) circle (.05);
\draw[fill, red] (2.15983, 1.74094) circle (.05);
\draw[fill, blue] (0.832251, 0.480501) circle (.05);
\draw[fill, red] (2.98326, 1.44334) circle (.05);
\draw[fill, red] (1.56807, 1.49919) circle (.05);
\draw[fill, red] (2.93067, 0.405813) circle (.05);
\draw[fill, blue] (2.39239, -0.251412) circle (.05);
\draw[fill, blue] (0.817638, -0.0518875) circle (.05);
\draw[fill, blue] (2.71824, 0.0662474) circle (.05);
\draw[fill, blue] (0.961215, -0.335126) circle (.05);
\draw[fill, red] (2.71824, 1.67563) circle (.05);
\draw[fill, blue] (1.23547, -0.531586) circle (.05);
\draw[fill, red] (2.3559, 1.75727) circle (.05);
\draw[fill, blue] (0.912185, 0.69395) circle (.05);
\draw[fill, red] (3.05154, 1.24748) circle (.05);
\draw[fill, red] (1.187, 1.13486) circle (.05);
\draw[fill, red] (3.05831, 0.809213) circle (.05);
\draw[fill, blue] (2.12003, -0.421092) circle (.05);
\draw[fill, blue] (0.794535, 0.232134) circle (.05);
\draw[fill, blue] (2.87442, 0.302068) circle (.05);
\draw[fill, blue] (1.09805, -0.448887) circle (.05);
\draw[fill, red] (2.52559, 1.73191) circle (.05);
\draw[fill, blue] (1.6437, -0.573073) circle (.05);
\draw[fill, red] (1.94267, 1.67563) circle (.05);
\draw[fill, blue] (1.32362, -0.554705) circle (.05);
\draw[fill, red] (2.25775, 1.74909) circle (.05);
\draw[fill, blue] (0.874793, 0.594101) circle (.05);
\draw[fill, red] (3.02077, 1.33574) circle (.05);
\draw[fill, red] (1.0602, 0.962485) circle (.05);
\draw[fill, red] (3.07591, 0.985537) circle (.05);
\draw[fill, blue] (2.2875, -0.316754) circle (.05);
\draw[fill, blue] (0.80924, 0.0513545) circle (.05);
\draw[fill, blue] (2.61333, -0.0360225) circle (.05);
\draw[fill, blue] (0.918779, -0.251412) circle (.05);
\draw[fill, red] (2.79132, 1.61157) circle (.05);
\draw[fill, blue] (1.91041, -0.510127) circle (.05);
\draw[fill, red] (1.71527, 1.56852) circle (.05);
\draw[fill, blue] (1.15983, -0.486066) circle (.05);
\draw[fill, red] (2.44715, 1.74363) circle (.05);
\draw[fill, blue] (0.967424, 0.794168) circle (.05);
\draw[fill, red] (3.06115, 1.14421) circle (.05);
\draw[fill, red] (1.35194, 1.3037) circle (.05);
\draw[fill, red] (3.00562, 0.61932) circle (.05);
\draw[fill, red] (1.1227, 1.04745) circle (.05);
\draw[fill, red] (3.06698, 0.896057) circle (.05);
\draw[fill, blue] (2.20981, -0.365155) circle (.05);
\draw[fill, blue] (0.802629, 0.132628) circle (.05);
\draw[fill, blue] (2.4883, -0.157908) circle (.05);
\draw[fill, blue] (0.86368, -0.142717) circle (.05);
\draw[fill, red] (2.89257, 1.52282) circle (.05);
\draw[fill, blue] (1.74324, -0.549581) circle (.05);
\draw[fill, red] (1.85491, 1.63429) circle (.05);
\draw[fill, blue] (1.42249, -0.560371) circle (.05);
\draw[fill, red] (2.15475, 1.72505) circle (.05);
\draw[fill, blue] (0.852154, 0.49196) circle (.05);
\draw[fill, red] (2.96865, 1.4118) circle (.05);
\draw[fill, red] (1.55143, 1.4648) circle (.05);
\draw[fill, red] (2.92498, 0.424388) circle (.05);
\draw[fill, red] (1.00974, 0.870938) circle (.05);
\draw[fill, red] (3.06808, 1.06968) circle (.05);
\draw[fill, blue] (2.36302, -0.256979) circle (.05);
\draw[fill, blue] (0.830668, -0.0249767) circle (.05);
\draw[fill, blue] (2.73404, 0.11158) circle (.05);
\draw[fill, blue] (0.997732, -0.347857) circle (.05);
\draw[fill, red] (2.66717, 1.67563) circle (.05);
\draw[fill, blue] (2.55004, -0.0977233) circle (.05);
\draw[fill, blue] (0.89154, -0.19768) circle (.05);
\draw[fill, red] (2.84043, 1.56852) circle (.05);
\draw[fill, blue] (1.81948, -0.531587) circle (.05);
\draw[fill, red] (1.79005, 1.60375) circle (.05);
\draw[fill, blue] (1.5454, -0.567432) circle (.05);
\draw[fill, red] (2.03394, 1.69691) circle (.05);
\draw[fill, blue] (0.821431, 0.353438) circle (.05);
\draw[fill, red] (1.42698, 1.3643) circle (.05);
\draw[fill, red] (2.97406, 0.543263) circle (.05);
\draw[fill, red] (1.19487, 1.1295) circle (.05);
\draw[fill, red] (1.19258, 1.12675) circle (.05);
\draw[fill, red] (3.04854, 0.801101) circle (.05);
\draw[fill, blue] (2.12859, -0.405845) circle (.05);
\draw[fill, blue] (0.810183, 0.221689) circle (.05);
\draw[fill, blue] (2.84949, 0.284339) circle (.05);
\draw[fill, blue] (1.09071, -0.434447) circle (.05);
\draw[fill, red] (2.52858, 1.7221) circle (.05);
\draw[fill, blue] (2.42039, -0.211588) circle (.05);
\draw[fill, blue] (0.846171, -0.0802097) circle (.05);
\draw[fill, red] (2.93234, 1.4648) circle (.05);
\draw[fill, blue] (1.66439, -0.557579) circle (.05);
\draw[fill, red] (1.92043, 1.66413) circle (.05);
\draw[fill, blue] (1.66812, -0.557108) circle (.05);
\draw[fill, blue] (1.30035, -0.531586) circle (.05);
\draw[fill, red] (2.28439, 1.73786) circle (.05);
\draw[fill, blue] (0.902881, 0.635872) circle (.05);
\draw[fill, red] (3.01966, 1.283) circle (.05);
\draw[fill, blue] (1.48408, -0.563869) circle (.05);
\draw[fill, red] (2.09337, 1.71062) circle (.05);
\draw[fill, blue] (0.837165, 0.424281) circle (.05);
\draw[fill, red] (1.48399, 1.41031) circle (.05);
\draw[fill, red] (2.95044, 0.488363) circle (.05);
\draw[fill, red] (1.28303, 1.22671) circle (.05);
\draw[fill, red] (3.02304, 0.697274) circle (.05);
\draw[fill, blue] (2.01311, -0.460407) circle (.05);
\draw[fill, blue] (2.77796, 0.175618) circle (.05);
\draw[fill, blue] (1.01038, -0.39235) circle (.05);
\draw[fill, red] (2.60767, 1.70829) circle (.05);
\draw (-8.66025, 3.5-.25) -- (8.66025,3.5-4.580125);
\end{tikzpicture}

%% file: Kleinfigure.tex
\begin{tikzpicture}[scale=.5]
\draw[fill] (-13.9248, -7.72768) circle (.05);
\draw[fill] (-13.9248, -2.21173) circle (.05);
\draw[fill] (13.9248, -7.72768) circle (.05);
\draw[fill] (-13.9248, 7.72768) circle (.05);
\draw[fill] (13.9248, -2.21173) circle (.05);
\draw[fill] (-13.9248, 2.21173) circle (.05);
\draw[fill] (13.9248, 7.72768) circle (.05);
\draw[fill] (-3.98539, -12.1511) circle (.05);
\draw[fill] (13.9248, 2.21173) circle (.05);
\draw[fill] (-6.59969, -9.53683) circle (.05);
\draw[fill] (-10.3775, -5.75906) circle (.05);
\draw[fill] (3.98539, -12.1511) circle (.05);
\draw[fill] (-3.98539, 12.1511) circle (.05);
\draw[fill] (-8.40885, -7.72768) circle (.05);
\draw[fill] (6.59969, -9.53683) circle (.05);
\draw[fill] (-6.59969, 9.53683) circle (.05);
\draw[fill] (10.3775, -5.75906) circle (.05);
\draw[fill] (-10.3775, 5.75906) circle (.05);
\draw[fill] (3.98539, 12.1511) circle (.05);
\draw[fill] (8.40885, -7.72768) circle (.05);
\draw[fill] (-8.40885, 7.72768) circle (.05);
\draw[fill] (6.59969, 9.53683) circle (.05);
\draw[fill] (6.64807e-12, -11.1668) circle (.05);
\draw[fill] (10.3775, 5.75906) circle (.05);
\draw[fill] (-3.98539, -10.1825) circle (.05);
\draw[fill] (-11.1668, -1.99754e-13) circle (.05);
\draw[fill] (8.40885, 7.72768) circle (.05);
\draw[fill] (-2.10858, -10.646) circle (.05);
\draw[fill] (-10.3775, -2.21173) circle (.05);
\draw[fill] (6.64807e-12, 11.1668) circle (.05);
\draw[fill] (-9.32074, -5.17262) circle (.05);
\draw[fill] (3.98539, -10.1825) circle (.05);
\draw[fill] (-3.98539, 10.1825) circle (.05);
\draw[fill] (11.1668, -1.99754e-13) circle (.05);
\draw[fill] (-9.85967, -3.66255) circle (.05);
\draw[fill] (2.10858, -10.646) circle (.05);
\draw[fill] (-2.10858, 10.646) circle (.05);
\draw[fill] (10.3775, -2.21173) circle (.05);
\draw[fill] (-10.3775, 2.21173) circle (.05);
\draw[fill] (9.32074, -5.17262) circle (.05);
\draw[fill] (-9.32074, 5.17262) circle (.05);
\draw[fill] (3.98539, 10.1825) circle (.05);
\draw[fill] (-7.04652, -7.72768) circle (.05);
\draw[fill] (9.85967, -3.66255) circle (.05);
\draw[fill] (-9.85967, 3.66255) circle (.05);
\draw[fill] (2.10858, 10.646) circle (.05);
\draw[fill] (10.3775, 2.21173) circle (.05);
\draw[fill] (-5.84084, -8.44025) circle (.05);
\draw[fill] (9.32074, 5.17262) circle (.05);
\draw[fill] (-3.98539, -9.53683) circle (.05);
\draw[fill] (-10.3775, -0.633014) circle (.05);
\draw[fill] (7.04652, -7.72768) circle (.05);
\draw[fill] (-7.04652, 7.72768) circle (.05);
\draw[fill] (9.85967, 3.66255) circle (.05);
\draw[fill] (-4.9697, -8.9551) circle (.05);
\draw[fill] (-10.3775, 0.633014) circle (.05);
\draw[fill] (-7.87564, -6.52957) circle (.05);
\draw[fill] (5.84083, -8.44025) circle (.05);
\draw[fill] (-5.84084, 8.44025) circle (.05);
\draw[fill] (-8.9551, -4.9697) circle (.05);
\draw[fill] (3.98539, -9.53683) circle (.05);
\draw[fill] (-3.98539, 9.53683) circle (.05);
\draw[fill] (10.3775, -0.633014) circle (.05);
\draw[fill] (7.04652, 7.72768) circle (.05);
\draw[fill] (-1.53056, -10.1825) circle (.05);
\draw[fill] (-8.40885, -5.75906) circle (.05);
\draw[fill] (4.9697, -8.9551) circle (.05);
\draw[fill] (-4.9697, 8.95508) circle (.05);
\draw[fill] (10.3775, 0.633014) circle (.05);
\draw[fill] (7.87564, -6.52957) circle (.05);
\draw[fill] (-7.87564, 6.52957) circle (.05);
\draw[fill] (5.84083, 8.4404) circle (.05);
\draw[fill] (-0.438105, -10.1825) circle (.05);
\draw[fill] (8.9551, -4.9697) circle (.05);
\draw[fill] (-8.9551, 4.96971) circle (.05);
\draw[fill] (3.98539, 9.53683) circle (.05);
\draw[fill] (1.53056, -10.1825) circle (.05);
\draw[fill] (-6.59969, -7.72768) circle (.05);
\draw[fill] (-9.70764, -3.04699) circle (.05);
\draw[fill] (-1.53056, 10.1825) circle (.05);
\draw[fill] (8.40895, -5.75906) circle (.05);
\draw[fill] (-8.40885, 5.75913) circle (.05);
\draw[fill] (4.9697, 8.95537) circle (.05);
\draw[fill] (0.438002, -10.1825) circle (.05);
\draw[fill] (-7.18143, -7.18143) circle (.05);
\draw[fill] (7.87564, 6.52937) circle (.05);
\draw[fill] (-2.66755, -9.77665) circle (.05);
\draw[fill] (-9.85518, -2.20858) circle (.05);
\draw[fill] (-0.438105, 10.181) circle (.05);
\draw[fill] (8.9551, 4.96965) circle (.05);
\draw[fill] (-3.98535, -9.30638) circle (.05);
\draw[fill] (-10.1076, -0.849385) circle (.05);
\draw[fill] (1.53056, 10.1825) circle (.05);
\draw[fill] (6.59963, -7.72768) circle (.05);
\draw[fill] (-6.59969, 7.7276) circle (.05);
\draw[fill] (9.70771, -3.04699) circle (.05);
\draw[fill] (-9.70764, 3.04701) circle (.05);
\draw[gray, opacity=.5, thin] (0, -10) -- (0, 10);
\draw[gray, opacity=.5, thin] (-10, 0) -- (10, 0);
\draw[gray, opacity=.5, thin] (2.60065, -9.65591) -- (-6.81781, 7.31556);
\draw[gray, opacity=.5, thin] (5.17956, -8.55407) -- (-9.83053, 1.83323);
\draw[gray, opacity=.5, thin] (2.60065, 9.65591) -- (-6.81781, -7.31556);
\draw[gray, opacity=.5, thin] (6.81781, 7.31556) -- (-2.60065, -9.65591);
\draw[gray, opacity=.5, thin] (9.83053, 1.83323) -- (-5.17956, -8.55407);
\draw[gray, opacity=.5, thin] (5.17956, 8.55407) -- (-9.83053, -1.83323);
\draw[gray, opacity=.5, thin] (6.81781, -7.31556) -- (-2.60065, 9.65591);
\draw[gray, opacity=.5, thin] (7.61769, -6.47848) -- (-9.51201, -3.08574);
\draw[gray, opacity=.5, thin] (9.83053, -1.83323) -- (-5.17956, 8.55407);
\draw[gray, opacity=.5, thin] (9.51201, -3.08574) -- (-7.61769, -6.47848);
\draw (-1.71397, -9.85202) -- (-9.26699, 3.75805);
\draw[gray, opacity=.5, thin] (7.61769, 6.47848) -- (-9.51201, 3.08574);
\draw (-5.89887, -8.07486) -- (-8.11056, 5.84969);
\draw[gray, opacity=.5, thin] (9.51201, 3.08574) -- (-7.61769, 6.47848);
\draw (9.26699, 3.75805) -- (1.71397, -9.85202);
\draw (-1.71397, 9.85202) -- (-9.26699, -3.75805);
\draw (8.11056, 5.84969) -- (5.89887, -8.07486);
\draw (-5.89887, 8.07486) -- (-8.11056, -5.84969);
\draw (9.26699, -3.75805) -- (1.71397, 9.85202);
\draw (3.48554, -9.37289) -- (-8.35964, -5.48785);
\draw (8.11056, -5.84969) -- (5.89887, 8.07486);
\draw (-0.741589, -9.97246) -- (-9.07156, -4.20795);
\draw (-4.75387, -8.79777) -- (-9.98072, 0.62069);
\draw (8.35964, -5.48785) -- (-3.48554, -9.37289);
\draw (3.48554, 9.37289) -- (-8.35964, 5.48785);
\draw (-3.17904, -9.48123) -- (-9.71541, -2.3687);
\draw (9.07156, -4.20795) -- (0.741589, -9.97246);
\draw (-0.741589, 9.97246) -- (-9.07156, 4.20795);
\draw (9.98072, 0.62069) -- (4.75387, -8.79777);
\draw (-4.75387, 8.79777) -- (-9.98072, -0.62069);
\draw (8.35964, 5.48785) -- (-3.48554, 9.37289);
\draw (9.71541, -2.3687) -- (3.17904, -9.48123);
\draw (-3.17904, 9.48123) -- (-9.71541, 2.3687);
\draw (9.07156, 4.20795) -- (0.741589, 9.97246);
\draw (4.45042, -8.9551) -- (-4.45042, -8.9551);
\draw (9.98072, -0.62069) -- (4.75387, 8.79777);
\draw (0.433894, -9.99058) -- (-7.10027, -7.04174);
\draw (-8.9551, -0) -- (-8.9551, -0);
\draw (9.71541, 2.3687) -- (3.17904, 9.48123);
\draw (2.02151, -9.79354) -- (-5.60192, -8.28363);
\draw (-8.52069, -5.2343) -- (-9.91445, 1.30526);
\draw (4.45042, 8.9551) -- (-4.45042, 8.9551);
\draw (-6.52185, -7.5806) -- (-9.88326, -1.52355);
\draw (7.10027, -7.04174) -- (-0.433894, -9.99058);
\draw (0.433894, 9.99058) -- (-7.10027, 7.04174);
\draw (-7.83322, -6.21617) -- (-9.99179, -0.405252);
\draw (5.60192, -8.28363) -- (-2.02151, -9.79354);
\draw (2.02151, 9.79354) -- (-5.60192, 8.28363);
\draw (9.91445, 1.30526) -- (8.52069, -5.2343);
\draw (-8.52069, 5.2343) -- (-9.91445, -1.30526);
\draw (9.88326, -1.52355) -- (6.52185, -7.5806);
\draw (-6.52185, 7.5806) -- (-9.88326, 1.52355);
\draw (7.10027, 7.04174) -- (-0.433894, 9.99058);
\draw (-4.27993, -9.03782) -- (-8.60567, -5.09338);
\draw (9.99179, -0.405252) -- (7.83322, -6.21617);
\draw (-7.83322, 6.21617) -- (-9.99179, 0.405252);
\draw (5.60192, 8.28363) -- (-2.02151, 9.79354);
\draw (9.91445, -1.30526) -- (8.52069, 5.2343);
\draw (-3.69046, -9.29411) -- (-7.39764, -6.72866);
\draw (9.88326, 1.52355) -- (6.52185, 7.5806);
\draw (-1.39672, -9.90198) -- (-6.06415, -7.95148);
\draw (-9.43395, -3.31673) -- (-9.76713, 2.14548);
\draw (8.60567, -5.09338) -- (4.27993, -9.03782);
\draw (-4.27993, 9.03782) -- (-8.60567, 5.09338);
\draw (9.99179, 0.405252) -- (7.83322, 6.21617);
\draw (-2.84968, -9.58537) -- (-6.62619, -7.48956);
\draw (-9.43395, 3.31673) -- (-9.76713, -2.14548);
\draw (-6.178, -7.86335) -- (-8.87181, -4.61422);
\draw (7.39764, -6.72867) -- (3.69046, -9.29411);
\draw (-3.69048, 9.2941) -- (-7.39763, 6.72867);
\draw (-7.48957, -6.62619) -- (-9.58537, -2.84968);
\draw (6.06415, -7.95148) -- (1.39672, -9.90198);
\draw (-1.39672, 9.90198) -- (-6.06415, 7.95148);
\draw (9.76713, 2.14548) -- (9.43395, -3.31673);
\draw (8.60567, 5.09338) -- (4.27993, 9.03782);
\draw (0.913994, -9.95814) -- (-3.80113, -9.2494);
\draw (-7.00368, -7.13782) -- (-9.18656, -3.95058);
\draw (6.62617, -7.48958) -- (2.84969, -9.58537);
\draw (-2.84974, 9.58535) -- (-6.62616, 7.4896);
\draw (9.76713, -2.14548) -- (9.43395, 3.31673);
\draw (8.87182, -4.6142) -- (6.17799, -7.86336);
\draw (-6.17799, 7.86336) -- (-8.87181, 4.61421);
\draw (7.39793, 6.72835) -- (3.68991, 9.29433);
\draw (1.50807, -9.88563) -- (-2.35159, -9.71957);
\draw (9.58537, -2.84968) -- (7.48956, -6.62619);
\draw (-7.48956, 6.62619) -- (-9.58537, 2.84968);
\draw (6.06415, 7.95148) -- (1.39672, 9.90198);
\draw (3.80113, -9.2494) -- (-0.913994, -9.95814);
\draw (-5.03678, -8.63891) -- (-7.74428, -6.32662);
\draw (-8.82493, -4.70326) -- (-9.92974, -1.18336);
\draw (0.913992, 9.95814) -- (-3.80113, 9.2494);
\draw (9.18676, -3.95013) -- (7.00342, -7.13808);
\draw (-7.0035, 7.138) -- (-9.18662, 3.95045);
\draw (6.62687, 7.48897) -- (2.84855, 9.5857);
\draw (2.35148, -9.7196) -- (-1.50816, -9.88562);
\draw (-5.7275, -8.1973) -- (-8.1973, -5.7275);
\draw (8.87162, 4.61458) -- (6.17857, 7.86291);
\draw (-1.03359, -9.94644) -- (-4.16132, -9.09304);
\draw (-9.35532, -3.53241) -- (-9.96811, -0.798038);
\draw (1.50062, 9.88677) -- (-2.34439, 9.72131);
\draw (9.58535, 2.84974) -- (7.48966, 6.62608);
\draw (-2.45516, -9.69393) -- (-5.32176, -8.46634);
\draw (-9.68391, -2.49438) -- (-9.96438, 0.843255);
\draw (3.80113, 9.2494) -- (-0.913994, 9.95814);
\draw (7.7441, -6.32684) -- (5.03694, -8.63882);
\draw (-5.0371, 8.63873) -- (-7.74411, 6.32683);
\draw (9.92978, -1.18302) -- (8.82478, -4.70353);
\draw (-8.8249, 4.7033) -- (-9.92974, 1.18335);
\draw (0,0) circle (10);
\end{tikzpicture}

%% file: main.bib
@article{Reading06,
  author  = {Reading, Nathan},
  title   = {Cambrian lattices},
  journal = {Advances in Mathematics},
  volume  = {205},
  number  = {2},
  year    = {2006},
  pages   = {313--353},
  doi     = {10.1016/j.aim.2005.07.010}
}

@article{ReadingSpeyer09,
  author  = {Reading, Nathan and Speyer, David E.},
  title   = {Cambrian fans},
  journal = {Journal of the European Mathematical Society},
  volume  = {11},
  number  = {2},
  year    = {2009},
  pages   = {407--447},
  doi     = {10.4171/JEMS/155}
}

@incollection{Bjorner84,
  author    = {Bj{\"o}rner, Anders},
  title     = {Orderings of {Coxeter} groups},
  booktitle = {Combinatorics and Algebra},
  series    = {Contemporary Mathematics},
  volume    = {34},
  publisher = {American Mathematical Society},
  address   = {Providence, RI},
  year      = {1984},
  pages     = {175--195}
}

@incollection{Reading12,
  author    = {Reading, Nathan},
  title     = {From the {Tamari} lattice to {Cambrian} lattices and beyond},
  booktitle = {Associahedra, {Tamari} Lattices and Related Structures},
  editor    = {M{\"u}ller-Hoissen, Folkert and Pallo, Jean Marcel and Stasheff, Jim},
  series    = {Progress in Mathematics},
  volume    = {299},
  publisher = {Birkh{\"a}user},
  address   = {Basel},
  year      = {2012},
  pages     = {293--322},
  doi       = {10.1007/978-3-0348-0405-9_15}
}

@article{BS72,
  author  = {Brieskorn, Egbert and Saito, Kyoji},
  title   = {Artin-Gruppen und Coxeter-Gruppen},
  journal = {Inventiones Mathematicae},
  volume  = {17},
  year    = {1972},
  pages   = {245--271},
  doi     = {10.1007/BF01406235}
}

@article {barkley:2024combinatorial,
    AUTHOR = {Barkley, Grant T. and Speyer, David E.},
     TITLE = {Combinatorial descriptions of biclosed sets in affine type},
   JOURNAL = {Comb. Theory},
  FJOURNAL = {Combinatorial Theory},
    VOLUME = {4},
      YEAR = {2024},
    NUMBER = {2},
     PAGES = {Paper No. 7, 33},
      ISSN = {2766-1334},
   MRCLASS = {20F55 (06B23 17B22)},
  MRNUMBER = {4807146},
}

@article{barkley:2023affine,
  title={Affine extended weak order is a lattice},
  author={Barkley, Grant T and Speyer, David E},
  journal={arXiv preprint arXiv:2311.05737},
  year={2023},
  note={To appear in {\textit {Transactions of the AMS}}},
  url={https://doi.org/10.1090/tran/9706}
}

@article {dyer:1994quotients,
    AUTHOR = {Dyer, M. J.},
     TITLE = {Quotients of twisted {B}ruhat orders},
   JOURNAL = {J. Algebra},
  FJOURNAL = {Journal of Algebra},
    VOLUME = {163},
      YEAR = {1994},
    NUMBER = {3},
     PAGES = {861--879},
      ISSN = {0021-8693,1090-266X},
   MRCLASS = {20F55 (06A07)},
  MRNUMBER = {1265869},
MRREVIEWER = {Robert\ B\'edard},
       DOI = {10.1006/jabr.1994.1049},
       URL = {https://doi.org/10.1006/jabr.1994.1049},
}

@article {dyer:2019weak,
    AUTHOR = {Dyer, Matthew},
     TITLE = {On the weak order of {C}oxeter groups},
   JOURNAL = {Canad. J. Math.},
  FJOURNAL = {Canadian Journal of Mathematics. Journal Canadien de
              Math\'ematiques},
    VOLUME = {71},
      YEAR = {2019},
    NUMBER = {2},
     PAGES = {299--336},
      ISSN = {0008-414X,1496-4279},
   MRCLASS = {20F55 (06B23 17B22)},
  MRNUMBER = {3943754},
MRREVIEWER = {Jean\ Fromentin},
       DOI = {10.4153/cjm-2017-059-0},
       URL = {https://doi.org/10.4153/cjm-2017-059-0},
}

@misc{BarkleyDefant25,
  author        = {Barkley, Grant and Defant, Colin},
  title         = {The Affine {T}amari Lattice},
  year          = {2025},
  eprint        = {2502.07198},
  archivePrefix = {arXiv},
  primaryClass  = {math.CO},
  doi           = {10.48550/arXiv.2502.07198}
}

@article {wang:2019infinite,
    AUTHOR = {Wang, Weijia},
     TITLE = {Infinite reduced words, lattice property and braid graph of
              affine {W}eyl groups},
   JOURNAL = {J. Algebra},
  FJOURNAL = {Journal of Algebra},
    VOLUME = {536},
      YEAR = {2019},
     PAGES = {170--214},
      ISSN = {0021-8693,1090-266X},
   MRCLASS = {20F55 (05E15)},
  MRNUMBER = {3989614},
MRREVIEWER = {Himmet\ Can},
       DOI = {10.1016/j.jalgebra.2019.06.032},
       URL = {https://doi.org/10.1016/j.jalgebra.2019.06.032},
}

@article{DyerHohlwegRipoll,
 author = {Dyer, Matthew and Hohlweg, Christophe and Ripoll, Vivien},
 title = {Imaginary cones and limit roots of infinite {Coxeter} groups},
 fjournal = {Mathematische Zeitschrift},
 journal = {Math. Z.},
 issn = {0025-5874},
 volume = {284},
 number = {3-4},
 pages = {715--780},
 year = {2016},
 language = {English},
 doi = {10.1007/s00209-016-1671-4},
 zbMATH = {6657022},
 Zbl = {1397.20048}
}

@article{HohlwegLabbe,
 author = {Hohlweg, Christophe and Labb{\'e}, Jean-Philippe},
 title = {On inversion sets and the weak order in {Coxeter} groups},
 fjournal = {European Journal of Combinatorics},
 journal = {Eur. J. Comb.},
 issn = {0195-6698},
 volume = {55},
 pages = {1--19},
 year = {2016},
 language = {English},
 doi = {10.1016/j.ejc.2016.01.002},
 zbMATH = {6569059},
 Zbl = {1333.05143}
}

@phdthesis{BarkleyThesis,
  title        = {Infinite root systems in algebra and geometry},
  author       = {Grant T. Barkley},
  year         = 2025,
  month        = {April},
  address      = {Cambridge, MA},
  school       = {Harvard University},
  type         = {Ph{D} thesis}
}

@phdthesis{LabbeThesis,
  title        = {Polyhedral {C}ombinatorics of {C}oxeter {G}roups},
  author       = {Jean-Philippe Labb{\'e}},
  year         = 2013,
  address      = {Berlin},
  school       = {Freien Universit\:at Berlin},
  type         = {Ph{D} thesis}
}

@article{DyerWang2021,
 author = {Dyer, Matthew and Wang, Weijia},
 title = {Oriented matroid structures from realized root systems},
 fjournal = {Journal of Algebraic Combinatorics},
 journal = {J. Algebr. Comb.},
 issn = {0925-9899},
 volume = {53},
 number = {4},
 pages = {945--971},
 year = {2021},
 language = {English},
 doi = {10.1007/s10801-020-00949-0},
 zbMATH = {7369370},
 Zbl = {1467.05026}
}

@misc{BDHMMS25,
  title         = {Extended Weak Order for the Rank 3 Universal {C}oxeter Group},
  author        = {Barkley, Grant and Defant, Colin and Hersh, Patricia and McCammond, Jon and McConville, Thomas and Speyer, David E.},
  year          = {2025},
  journal       = {Mathematische Zeitschrift},
  note          = {To appear in {\em Math. Z.}},
  eprint        = {2509.00871},
  archivePrefix = {arXiv},
  primaryClass  = {math.CO},
  doi           = {10.48550/arXiv.2509.00871}
}
